\documentclass[12pt]{amsart}
\usepackage{amsmath,amsfonts,euscript,dsfont,amscd,amsthm,amssymb,upref,graphics}
\usepackage[all]{xy}

\theoremstyle{definition}

\swapnumbers

\theoremstyle{plain}

\newtheorem{theorem}{Theorem}[section]
\newtheorem{proposition}[theorem]{Proposition}
\newtheorem{lemma}[theorem]{Lemma}
\newtheorem{corollary}[theorem]{Corollary}
\newtheorem{observation}[theorem]{Observation}
\newtheorem{paragItalic}[theorem]{}

\theoremstyle{definition}

\newtheorem{definition}[theorem]{Definition}

\newtheorem{example}[theorem]{Example}
\newtheorem{examples}[theorem]{Examples}
	
\newtheorem{notation}[theorem]{Notation}

\newtheorem{remark}[theorem]{Remark}

\theoremstyle{remark}

\newtheorem*{smallremark}{Remark}

{\begin{enumerate}\setlength{\itemsep}{#1}}{\end{enumerate}}
\newenvironment{enumerata}%
{\begin{enumerate}

}{\end{enumerate}}
{\begin{enumerata}\setlength{\itemsep}{#1}}{\end{enumerata}}

\newcommand{\Aut}{	\operatorname{{\rm Aut}}}
\newcommand{\Spec}{	\operatorname{{\rm Spec}}}

\newcommand{\supp}{	\operatorname{{\rm supp}}}
\newcommand{\image}{	\operatorname{{\rm im}}}

\newcommand{\bideg}{	\operatorname{{\rm bideg}}}
\newcommand{\subdeg}{	\operatorname{{\rm subdeg}}}
\newcommand{\DEG}{	\operatorname{\text{\sc deg}}}
\newcommand{\UDEG}{	\operatorname{\text{\underline{\sc deg}}}}

\newcommand{\Frac}{	\operatorname{{\rm Frac}}}

\renewcommand{\div}{	\operatorname{{\rm div}}}

\newcommand{\Miss}{	\operatorname{{\rm Miss}}}
\newcommand{\Cont}{	\operatorname{{\rm Cont}}}

\newcommand{\dic}{	\operatorname{{\rm dic}}}

\newcommand{\Bir}{	\operatorname{{\rm Bir}}}

\newcommand{\Rec}{	\operatorname{{\rm Rec}}}

\newcommand{\Sw}{S_{\text{\rm w}}}

\newcommand{\Galg}{	\operatorname{\Gamma_{\!\text{\rm alg}}}}

\newcommand{\setspec}[2]{\big\{\,#1\, \mid \,#2\, \big\}}

\newcommand{\Integ}{\ensuremath{\mathbb{Z}}}
\newcommand{\Nat}{\ensuremath{\mathbb{N}}}

\newcommand{\Reals}{\ensuremath{\mathbb{R}}}
\newcommand{\aff}{\ensuremath{\mathbb{A}}}
\newcommand{\proj}{\ensuremath{\mathbb{P}}}
\newcommand{\bk}{{\ensuremath{\rm \bf k}}}

\newcommand{\kk}[1]{\bk^{[#1]}}

\newcommand{\bbV}{\ensuremath{\mathbb{V}}}

\newcommand{\Cgoth}{{\ensuremath{\mathfrak{C}}}}

\newcommand{\pgoth}{{\ensuremath{\mathfrak{p}}}}

\newcommand{\mgoth}{{\ensuremath{\mathfrak{m}}}}

\newcommand{\Aeul}{\EuScript{A}}

\newcommand{\Deul}{\EuScript{D}}

\newcommand{\Meul}{\EuScript{M}}

\newcommand{\Oeul}{\EuScript{O}}
\newcommand{\Peul}{\EuScript{P}}

\renewcommand{\epsilon}{\varepsilon}
\renewcommand{\phi}{\varphi}
\renewcommand{\emptyset}{\varnothing}

\newcommand{\BirA}{\Bir(\aff^2)}
\newcommand{\AutA}{\Aut(\aff^2)}

\newcommand{\rien}[1]{}

\CompileMatrices

\begin{document}
\renewcommand{\baselinestretch}{1.07}


\title[Field generators and birational endomorphisms of $\aff^2$]%
{Field generators in two variables\\ and birational endomorphisms of $\aff^2$}

\author{Pierrette Cassou-Nogu\`es}
\author{Daniel Daigle}

\dedicatory{In memory of Shreeram S.\ Abhyankar.}

\address{IMB, Universit\'e Bordeaux 1 \\
351 Cours de la lib\'eration, 33405, Talence Cedex, France}
\email{Pierrette.Cassou-nogues@math.u-bordeaux1.fr}

\address{Department of Mathematics and Statistics\\
University of Ottawa\\
Ottawa, Canada\ \ K1N 6N5}
\email{ddaigle@uottawa.ca}

\thanks{Research of the first author partially supported by Spanish grants MTM2010-21740-C02-01 and  MTM2010-21740-C02-02.}
\thanks{Research of the second author supported by grant RGPIN/104976-2010 from NSERC Canada.}

{\renewcommand{\thefootnote}{}
}

{\renewcommand{\thefootnote}{}
}

\begin{abstract}
This article is a survey of two subjects:
the first part is devoted to field generators in two variables,
and the second to birational endomorphisms of the affine plane.
Each one of these subjects originated in Abhyankar's seminar in Purdue University in the 1970s.
Note that the part on field generators is more than a survey, since it contains a considerable amount of new material.
\end{abstract}

\maketitle

This article is a survey of two subjects:
the first part of the paper is devoted to field generators in two variables,
and the second to birational endomorphisms of the affine plane.
Each one of these subjects originated in Abhyankar's seminar in Purdue University in the 1970s.
The authors of the present article were introduced to these questions by Peter Russell,
who participated in Abhyankar's seminar and who made early contributions to both problems.

As explained in Section~\ref{SEC:Introduction784y18372yh}, the two subjects are entangled one into the other.
It is therefore natural to present them together in a survey.
Note that Part~I is more than a survey, since it contains a considerable amount of new material
(see Section~\ref{SEC:Introduction784y18372yh}); and that Part~II is less than a survey, since it
restricts itself to certain particular aspects of the subject under consideration 
(see Section~\ref{IntroPartII78364174}).

\section*{Conventions}

The symbol ``$\subset$'' means strict inclusion of sets, ``$\setminus$'' means set difference, and $0 \in \Nat$.

\medskip
If $R$ is a subring of a ring $S$, the notation $S=R^{[n]}$ means that $S$ is $R$-isomorphic
to a polynomial algebra in $n$ variables over $R$.
If $L/K$ is a field extension, $L = K^{(n)}$ means that $L$ is a purely transcendental extension of $K$,
of transcendence degree $n$.
We write $\Frac R$ for the field of fractions of a domain $R$.

\medskip
If $\bk$ is a field and $A = \kk2$ (i.e., $A$ is a polynomial ring in two variables over $\bk$) then
a {\it coordinate system\/} of $A$ is an ordered pair $(X,Y) \in A \times A$ satisfying $A=\bk[X,Y]$.
We define $\Cgoth(A)$ to be the set of coordinate systems of $A$.
A {\it variable\/} of $A$ is an element $X \in A$ satisfying $A = \bk[X,Y]$ for some $Y$.

Consider $F \in A = \kk2$.  For each $\gamma = (X,Y) \in \Cgoth(A)$, we write $\deg_\gamma(F)$ for the total
degree of $F$ as a polynomial in $X,Y$ (let us agree that $\deg_\gamma(0) = -\infty$).
We set
$$
\deg_A(F) = \min \setspec{ \deg_\gamma(F) }{ \gamma \in \Cgoth(A) } .
$$

\medskip
Given a field $\bk$, we write $\aff^2 = \aff^2_\bk$ for the affine plane over $\bk$,
i.e., $\aff^2 = \Spec A$ for some $A = \kk2$.
By a {\it coordinate system\/} of $\aff^2$, we mean a coordinate system of $A$.
That is, a coordinate system of $\aff^2$ is an  $(X,Y) \in A \times A$ satisfying $A=\bk[X,Y]$.
The set of coordinate systems of $\aff^2$ is denoted $\Cgoth(\aff^2)$ or simply $\Cgoth$;
so $\Cgoth = \Cgoth(\aff^2) = \Cgoth(A)$.

\medskip
By ``curve'' we mean ``irreducible and reduced curve''.

\bigskip
\medskip
\noindent
\begin{minipage}{\textwidth}
\begin{center}
{\large\bf Part I: Field generators}
\end{center}

\bigskip
Throughout Part~I, the following convention about the base field $\bk$ is in effect:
\begin{itemize}
\item In all examples, $\bk$ is tacitly assumed to be an algebraically closed field of characteristic zero.
\item Everywhere else, $\bk$ denotes an arbitrary field unless the contrary is explicitly specified.
\end{itemize}
\end{minipage}

\section{Introduction}
\label {SEC:Introduction784y18372yh}

\begin{definition} 
Let $A = \kk2$ and $K = \Frac A$.
A {\it field generator of $A$} is an $F \in A$ satisfying $K = \bk(F,G)$ for some $G \in K$.
A {\it good field generator of $A$} is an $F \in A$ satisfying $K = \bk(F,G)$ for some $G \in A$.
A field generator which is not good is said to be {\it bad}.
\end{definition}

The notions of good and bad field generators are classical. 
The two fundamental articles on this subject were written by Russell in 1975 and 1977
(cf.\  \cite{Rus:fg} and \cite{Rus:fg2}).
The main results of those two papers will be explained in the course of the present article.
Bad field generators were once supposed not to exist, then two examples were given,
the first one by Jan \cite{JanThesis} (unpublished) in 1974, of degree 25,
and the second one by Russell \cite{Rus:fg2} in 1977 of degree 21.
There were no more examples until 2005, when the first author showed \cite{Cassou-BadFG}
that for any $N$ there exists a bad field generator $F$ of $A=\kk2$ such that $\deg_A(F) \ge N$.

\medskip
It is an open question to classify field generators.

\medskip
Throughout Part~I we shall use the convention that the notation ``$A \preceq B$'' means
that all of the following conditions are satisfied:
$$
A = \kk2, \quad B = \kk2, \quad A \subseteq B \quad \text{and} \quad \Frac A = \Frac B .
$$

Observe that if $F \in A \preceq B$, then
$F$ is a field generator of $A$ iff it is a field generator of $B$.
So the problem of classifying field generators is intertwined with that of describing all pairs $A \preceq B$,
or equivalently, with the problem of classifying birational endomorphisms of $\aff^2$. 
The latter problem is the subject of Part~II of the present paper,
and is a hard and interesting problem in its own right.  It therefore seems reasonable to keep those two problems 
separated, i.e., if our aim is to classify field generators, then we should primarily try to
classify {\it those field generators that are not composed with a birational endomorphism.}
A polynomial $F \in A=\kk2$ that is not composed with a birational endomorphism is said to be ``lean'':

\begin{definition} \label {823gdvcuuyv2trc6374rnmujfj}
Let $F \in A=\bk^{[2]}$.  We say that {\it $F$ is lean in $A$} if, for each $A'$ such that $F\in A' \preceq A$, there
holds $A'=A$.  We say that $F \in A$ {\it admits a lean factorization\/} if
there exists $A'\preceq A$ such that $F\in A'$ and $F$ is lean in $A'$. 
\end{definition}

The problem of classifying field generators contains the following subproblems:
\begin{enumerate}

\item[(i)] {\it Determine which field generators do not admit a lean factorization, and classify them.}

\item[(ii)] {\it Classify the field generators that are lean.}

\end{enumerate}
By composing the polynomials (ii) with all birational endomorphisms of $\aff^2$ one obtains
precisely all field generators {\it that admit a lean factorization}; then adding the polynomials (i)
to this set gives all field generators.
We regard (i) and (ii) as the most interesting components of the problem of classifying field generators.
There is, however, another aspect that is of crucial importance:
\begin{enumerate}
\item[(iii)] {\it Describe how field generators behave under birational extensions $A \preceq B$.}
\end{enumerate}
In some sense, (iii) is a theme that underlies the whole paper. 
Results \ref{difja;skdfj;aksd},  \ref{56fewf8r8t34kd223lwgfuio},  \ref{of82098u12bw9},
\ref{m93919gdqwd8719} and \ref{23h7232sqwadzcpydgb} are good illustrations of the type of theory
that (iii) calls for.  

Before discussing (i) and (ii), we need to define the notions of very good and very bad field generators.
We already noted that if $F \in A \preceq B$, then $F$ is a field generator of $A$ iff it is a field 
generator of $B$.
Moreover, if $F$ is a good field generator of $A$ then it is a good field generator of $B$
(and consequently, if it is a bad field generator of $B$ then it is a bad field generator of $A$).
However, it might happen that $F$ be a bad field generator of $A$ and  a good field generator of $B$.
These remarks suggest the following:

\begin{definition}   \label {Eir128372et12uhwjd}
Let $F \in A = \kk2$ be a field generator of $A$.
\begin{enumerate}

\item $F$ is a {\it very good\/} field generator of $A$ if it is a good field generator
of each $A'$ satisfying $F \in A' \preceq A$.

\item $F$ is a {\it very bad\/} field generator of $A$ if it is a bad field generator
of each $A'$ satisfying $A' \succeq A$.

\end{enumerate}
\end{definition}

Problem~(i) is partially solved by \ref{888dfgdzb5n3d2ff},
which asserts that {\it the field generators that  do not admit lean factorizations
are precisely the very good field generators}.
This is in fact the reason why we became interested in the concept of very good field generator.
The very good field generators are not yet classified, but
Sections \ref{SEC:VGVBFGs} and \ref{Leanbadfieldgenerators} give several results about them
(\ref{q23328r83yd74r9128}, \ref{last} and various examples and remarks).

Problem~(ii) is probably the hardest part of the whole question.
Although Section~\ref{Leanbadfieldgenerators} gives some results on this subject,
our understanding is still very incomplete.

\medskip
Most of the results given in Sections~\ref{SEC:PrelimsPartII}--\ref{SEC:VGVBFGs}
can be found in the article \cite{CassouDaigVGVBFG}.
However, most of the examples never appeared in the literature before.
All the material of Section~\ref{Leanbadfieldgenerators} is new.
Note in particular that \ref{p92839819ueg87gdywfe6} gives an example of a very bad field generator
that is also lean, and that no such example was known before.

\begin{remark}  \label {723te762gd1827}
If $\bk$ is an algebraically closed field of characteristic zero then
$F$ is a field generator of $A=\kk2$ if and only if it is a ``rational polynomial'' of $A$.
(By a {\it rational polynomial of $A$}, we mean an element $F \in A \setminus\bk$ such
that, for all but possibly finitely many $\lambda \in \bk$, $F-\lambda$ is irreducible
and the plane curve ``$F-\lambda=0$'' is rational.)
For this equivalence and analoguous results in positive characteristic, see \cite{Dai:GenRatPols}.
\end{remark}

\section{Dicriticals}
\label {SEC:PrelimsPartII}

\begin{definition}  \label {difp23qwjksd}
Given a field extension $L \subseteq M$, let $\bbV(M/L)$ be the set of valuation rings $R$
satisfying $L \subseteq R \subseteq M$, $\Frac R = M$ and $R \neq M$.

Given a pair $(F,A)$ such that $A = \kk2$ and $F \in A \setminus \bk$, define
$$
\bbV^\infty(F,A)
= \setspec{ R \in \bbV( K /\bk(F) ) }{  A \nsubseteq R } \quad \text{where $K = \Frac A$.}
$$
Then $\bbV^\infty(F,A)$ is a nonempty finite set which depends only on the pair $(\bk(F), A)$.
For each $R \in \bbV^\infty(F,A)$, let $\mgoth_R$ be the maximal ideal of $R$.
Let $R_1, \dots, R_t$ be the distinct elements of $\bbV^\infty(F,A)$
and $d_i = [ R_i/\mgoth_{R_i} : \bk(F) ]$ for $i = 1, \dots, t$.
Then we define
$$
\Delta(F,A) = [ d_1, \dots, d_t ] \quad\text{and}\quad \dic(F,A) = | \bbV^\infty(F,A) | = t
$$
where $[ d_1, \dots, d_t ]$ is an unordered $t$-tuple of positive integers.

Given $A = \kk2$ and $F \in A \setminus \bk$,
we call the elements of  $\bbV^\infty(F,A)$ the {\it dicriticals\/} of $(F,A)$, or of $F$ in $A$;
given $R \in \bbV^\infty(F,A)$, we call $[ R/\mgoth_R : \bk(F) ]$ the {\it degree of the dicritical\/}~$R$.
\end{definition}

Except for the notations,
our definitions of ``dicritical'' and of ``degree of dicritical'' are identical to those 
given by Abhyankar in \cite{Abh:DicDiv2010} (see the last sentence of page 92).
The  following fact is very useful for determining the degree list $\Delta(F,A)$ of an explicit polynomial:

\begin{lemma} \label {hf1g72wg1761878d}
Assume that $\bk$ is algebraically closed, let $A = \kk2$ and $F \in A\setminus\bk$.
Let $f : \aff^2 = \Spec A \to \aff^1 = \Spec\bk[F]$ be the morphism determined by the inclusion $\bk[F] \to A$.
Then there exists a (non unique) commutative diagram
\begin{equation}  \label {dkfjasodfla2}
\raisebox{5mm}{$\xymatrix{
\aff^2 \ar[d]_{f} \ar @<-.4ex> @{^{(}->}[r]  &  X \ar[d]^{\bar f}  \\
\aff^1 \ar @<-.4ex> @{^{(}->}[r] & \proj^1
}$}
\end{equation}
where $X$ is a nonsingular projective surface, the arrows ``$\hookrightarrow$'' are open immersions and
$\bar f$ is a morphism.
Among the irreducible components of $X \setminus \aff^2$, let $C_1, \dots, C_t \subset X$ be the curves 
that satisfy $\bar f(C_i) = \proj^1$, and for each $i=1,\dots, t$, let $d_i$ be the degree of the
surjective morphism $\bar f|_{C_i} : C_i \to \proj^1$.  Then $ \Delta(F,A) = [d_1, \dots, d_t] $.
\end{lemma}

\begin{proof}
We sketch the proof, and refer to \cite[2.3]{CassouDaigVGVBFG} for details.
For each $i = 1, \dots, t$, let $\xi_i \in X$ be the generic point of $C_i$.
Then the local rings $\Oeul_{X,\xi_i}$ are valuation rings,
$\bbV^\infty(F,A) = \setspec{ \Oeul_{X,\xi_i} }{ 1 \le i \le t }$ and, for each $i$,
the degree of $\bar f|_{C_i} : C_i \to \proj^1$ is equal to $[ K(C_i) : K(\proj^1) ] = 
[ \Oeul_{X,\xi_i} / \mgoth_{X,\xi_i} : \bk(F) ]$.
\end{proof}

We shall make tacit use of \ref{hf1g72wg1761878d} in all examples of the present paper.
In practice, we find a diagram~\eqref{dkfjasodfla2} by resolving the base points of the linear system
$\Lambda(F)$ on $\proj^2$ defined as
$ \Lambda(F) = \setspec{ V( a F^*(X,Y,Z) + b Z^n ) }{ (a:b) \in \proj^1 } $,
where $n = \deg F$ and $F^*(X,Y,Z) = Z^n F(X/Z,Y/Z)$ is the standard homogenization of $F$.

The following fact appears as ``$\text{\it GCD}( \deg \bar q_1, \dots \deg \bar q_r ) =1$''
in the proof of \cite[3.8]{Rus:fg},
and is also a special case of \cite[2.5]{CassouDaigVGVBFG}:

\begin{corollary}  \label {2384y872tqo8we1ftbs}
If $F$ is a field generator of $A = \kk2$ and $\Delta(F,A) = [d_1, \dots, d_t]$, then $\gcd(d_1, \dots, d_t) = 1$.
\end{corollary}

\begin{lemma}  \label {828dbvcdqscd51671w72}
Let $A = \kk2$, $F \in A\setminus\bk$ and $\Delta(F,A) = [d_1, \dots, d_t]$. Then
$$
\deg_A(F) \ge \sum_{i=1}^t d_i .
$$
\end{lemma}

\begin{proof}
Let $\gamma = (X,Y) \in \Cgoth(A)$; let us prove that 
\begin{equation}  \label {3e2q23bcf7atx1n2myhx2er62}
\textstyle   \deg_\gamma(F) \ge \sum_{i=1}^t d_i .
\end{equation}

We first consider the case where $\bk$ is an infinite field.
For each $\lambda \in \bk^*$, let $X_\lambda = X-\lambda Y \in \bk[X,Y] = A$. Since $A = \bk[X_\lambda,Y]$,
there exists a unique $G_\lambda(S,T) \in \bk[S,T]$ such that $F = G_\lambda(X_\lambda,Y)$.
Moreover, 
\begin{equation} \label {98dh24gjubnmzzq}
\deg_T( G_\lambda(S,T) ) \le \deg_\gamma(F).
\end{equation}
Consider the field $M = \bk(F,X_\lambda)$. Since $\Frac A = M(Y)$ and $Y$ is a root of the polynomial
$G_\lambda(X_\lambda,T) - F \in M[T]$, \eqref{98dh24gjubnmzzq} implies that 
$$
[\Frac A : \bk(F,X_\lambda) ] \le  \deg_\gamma(F), \quad \text{for every $\lambda \in \bk^*$.}
$$

Let $\bbV^\infty(F,A) = \{ R_1, \dots, R_t \}$.
For each $i=1,\dots,t$,
let $d_i = [R_i/\mgoth_{R_i} : \bk(F)]$ and let $v_i : (\Frac A)^* \to \Integ$ be the valuation such that $R_i = R_{v_i}$;
note that at most one element $\lambda_i \in \bk^*$ satisfies $v_i( (X/Y) - \lambda_i )>0$.
Since $\bk$ is infinite, we may choose
$\lambda \in \bk^*$ satisfying
$$
\text{$v_i( (X/Y) - \lambda )\le 0$ for all $i=1,\dots,t$.}
$$
We claim:
\begin{equation}  \label {09o9a09f9d8hasilsisd}
\text{$v_i( X_\lambda ) < 0$ for all $i=1,\dots,t$.}
\end{equation}
Indeed, let $i \in \{1, \dots, t\}$ and observe that $\min(v_i(X),v_i(Y))<0$.
If $v_i(X) \neq v_i(Y)$ then $v_i(X_\lambda) = \min(v_i(X),v_i(Y))<0$.
If $v_i(X)=v_i(Y)$ then $v_i(Y) < 0$
and $v_i( X_\lambda ) = v_i(Y) + v_i( (X/Y) - \lambda ) \le v_i(Y) < 0$.
This proves \eqref{09o9a09f9d8hasilsisd}.

Since $X_\lambda \in A$, $v(X_\lambda)\ge0$ for all valuations $v$ of $\Frac A$ over $\bk(F)$ other than $v_1, \dots, v_t$.
This and \eqref{09o9a09f9d8hasilsisd} implies that the divisor of poles of $X_\lambda$ is
$\div_\infty(X_\lambda) = \sum_{i=1}^t m_i R_i$  where $m_i = -v_i(X_\lambda) > 0$ for all $i$.
Consequently,
$$
\textstyle
\sum_{i=1}^t d_i \le  \sum_{i=1}^t m_i d_i = \deg\big( \div_\infty(X_\lambda) \big)
 =  [\Frac A : \bk(F,X_\lambda)] \le \deg_\gamma(F),
$$
so \eqref{3e2q23bcf7atx1n2myhx2er62} is true whenever $\bk$ is an infinite field.

Now drop the assumption on $\bk$ (so $\bk$ is now an arbitrary field).
Pick an indeterminate $\tau$ transcendental over $\Frac A$,
let $\hat\bk = \bk(\tau) = \bk^{(1)}$ and $\hat A = \hat\bk[X,Y] = \hat\bk^{[2]}$.
Since $F \in A \subset \hat A$, we may consider $\bbV^\infty(F,A)$ and $\bbV^\infty(F,\hat A)$.
It is easy to see that
\begin{equation} \label {3274dcb7bd4j9sty}
\text{$S \mapsto S \cap \Frac A$ is a surjective map $\bbV^\infty(F,\hat A) \to \bbV^\infty(F,A)$.}
\end{equation}
Let us adopt the temporary notation $\Delta^\sharp(F,A) = \sum_{i=1}^t d_i$ where $\Delta(F,A) = [d_1, \dots, d_t]$.
It's enough to show:
\begin{equation}  \label {9df09x9rhdh193y}
\Delta^\sharp(F,\hat A) \ge \Delta^\sharp(F,A).
\end{equation}
Indeed, we have $\deg_\gamma(F) \ge \Delta^\sharp(F,\hat A)$
by the first part of the proof and the fact that $\hat\bk$
is an infinite field, so if \eqref{9df09x9rhdh193y} is true then we are done.
To prove \eqref{9df09x9rhdh193y},
consider $S \in \bbV^\infty(F,\hat A)$, let $R = S \cap \Frac A$ and consider the field extensions:
$$
\xymatrix@R=1mm{
 & S/\mgoth_S  \\
\hat\bk(F) \ar @{-}[ur] && R/\mgoth_R  \ar @{-}[ul] \\
 & \bk(F) \ar @{-}[ul] \ar @{-}[ur]
}
$$
Since  $R/\mgoth_R$ is algebraic over $\bk(F)$ and 
$\hat\bk(F)$ is purely transcendental over $\bk(F)$, 
$R/\mgoth_R$ is linearly disjoint from $\hat\bk(F)$ over $\bk(F)$; thus
$[ S/\mgoth_S : \hat\bk(F) ] \ge [ R/\mgoth_R : \bk(F) ]$.
As this holds for each  $S \in \bbV^\infty(F,\hat A)$, \eqref{3274dcb7bd4j9sty} implies 
$$
\sum_{S \in \bbV^\infty(F,\hat A)} [ S/\mgoth_S : \hat\bk(F) ]\ \ \ge
\sum_{R \in \bbV^\infty(F,A)} [ R/\mgoth_R : \bk(F) ],
$$
which is exactly \eqref{9df09x9rhdh193y}. 
\end{proof}

\begin{observation}[{\cite[Rem.\ after 1.3]{Rus:fg}}] \label {82364r5gg658943k}
Let $F$ be a field generator of $A=\kk2$.
Then $F$ is a good field generator of $A$ if and only if \/\mbox{\rm ``$1$''} occurs in the list $\Delta(F,A)$.
\end{observation}

\begin{smallremark}
Recall that in all examples of Part~I, $\bk$ is assumed
to be an algebraically closed field of characteristic zero (see the introduction to Part~I).
The terms ``Newton polygon'' and ``Newton tree'' are sometimes used in the examples below.
The Newton polygon of a polynomial $\sum_{ij} a_{ij}X^iY^j$ is the convex hull in $\Reals^2$
of $\{0\} \cup \setspec{(i,j)}{ a_{ij} \neq 0 }$; the sides of that polygon that are not included 
in the axes of coordinates are called the ``faces'' of the Newton polygon.
See \cite{Cassou_Russellfest} for the notion of Newton tree. 
From the Newton tree at infinity of $F(X,Y)$, one can deduce the
genus of the curve ``$F(X,Y)=t$'' for general $t \in \bk$;
however, readers not familiar with these notions may ignore Newton trees altogether,
and use the well known genus formula (more is said about this in \ref{2u37vd81182ejfaeu}).
\end{smallremark}

\begin{example}
\label {2u37vd81182ejfaeu}
The first example of bad field generator was given by Jan
(A. Sathaye kindly gave us the equation of that polynomial).
Let  $A = \bk[X,Y] = \kk2$ and
$$F_J(X,Y)=Y(X^8Y^4-1)^2+3X^3Y^2(X^8Y^4-1)+3X^6Y^3+X \in A . $$
It has two points at infinity. 
The Newton polygon has two faces, one linking the point $(1,0)$ to the point $(16,9)$ with slope $3/5$
and the other one linking the point $(0,1)$ to $(16,9)$ with slope $1/2$.

Let $F_J(X,Y,Z) \in \bk[X,Y,Z]$ denote the standard homogenization of $F_J$.
At the point $[0:1:0]$, we have 
$$
F_J(X,1,Z)+tZ^{25}=(X^8-Z^{12})^2+3X^3(X^8-Z^{12})Z^8+3X^6Z^{16}+XZ^{24}+tZ^{25}.
$$
After the blowups: $X\to XZ$ and divide by $Z^{16}$,
$Z\to ZX$ and divide by $X^8$,
and $X\to XZ$ and divide by $Z^8$,
the strict transform of $F_J(X,1,Z)+tZ^{25}$ is
$$
(X^4-1)^2+3X^2(X^4-1)Z+3X^4Z^2+X^2Z^3+tXZ^2 .
$$
In view of \ref{hf1g72wg1761878d}, this shows that
we have four dicriticals of degree $2$, corresponding to the roots of $X^4-1$
(these four dicriticals are over the point $[0:1:0]$).
At the point $[1:0:0]$, we have 
$$F_J(1,Y,Z)+tZ^{25}=Y(Y^4-Z^{12})^2+3Y^2(Y^4-Z^{12})Z^8+3Y^3Z^{16}+Z^{24}+tZ^{25}$$
After the blowups: $Y\to YZ^2$ and divide by $Z^{18}$, $Z\to ZY$ and divide by $Y^6$, and $Y\to YZ^2$ and divide by $Z^6$,
we get 
$$
Y^3(1-Z^4)^2+3Y^2(1-Z^4)+3Y+1+tYZ^3 .
$$
After the change $Y\to Y-1$ we have
$Y^3+h(t)Z^3+\sum a_{\alpha_1,\alpha_2}Y^{\alpha_1}Z^{\alpha_2}$ with $\alpha_1+\alpha_2>3$,
and $h(t)$ is a polynomial in $t$ of degree $1$.
This proves that (over the point $[1:0:0]$) we have one dicritical of degree $3$. Then 
$$
\Delta(F_J,\bk[X,Y])=[2,2,2,2,3] .
$$
From the above computations one deduces the configuration of singularities at infinity,
from which one obtains the Newton tree at infinity of $F_J$ shown in Figure 1.
From that Newton tree---or directly from the  configuration of singularities at infinity and the genus formula---it
follows that (for general $t \in \bk$) the
plane curve ``$F_J(X,Y)=-t$'' is rational.
So $F_J$ is a rational polynomial and hence (\ref{723te762gd1827}) a field generator of $A=\bk[X,Y]$.
By \ref{82364r5gg658943k}, $F_J$ is a bad field generator of $A$.
\end{example}

\begin{figure}[ht]

\begin{center}

\includegraphics{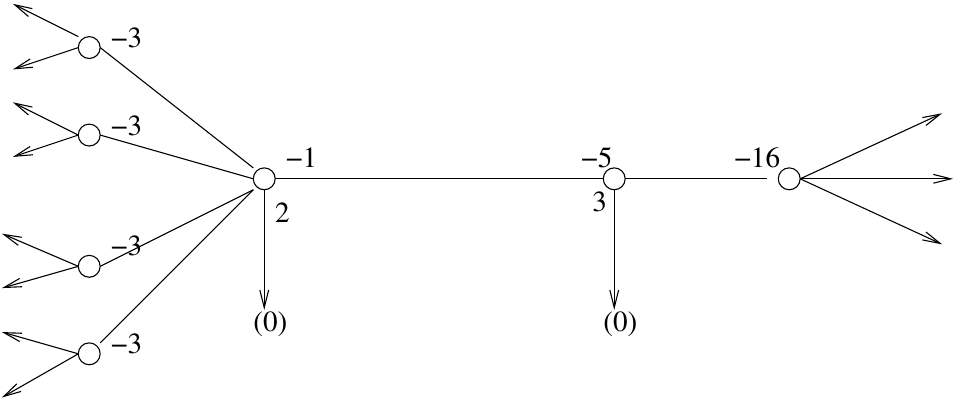}

\caption{}

\end{center}

\end{figure}  

The second example of bad field generator was given by Russell in \cite{Rus:fg2}, and is the following polynomial
of degree $21$:\footnote{There were some misprints in the polynomial given in \cite{Rus:fg2}.  The polynomial that is
displayed here is the correct one.}
\begin{equation*}
( Y^2(XY + 1)^4 + Y(2XY + 1)(XY + 1) + 1) (Y(XY + 1)^5 +2XY(XY + 1)^2 +X)\,.
\end{equation*}
We shall refer to this polynomial as ``Russell's polynomial''.  It is a bad field generator of $\bk[X,Y]$, where
$\bk$ is an arbitrary field.
The same paper contains the following fact, valid for an arbitrary field $\bk$
(we use the notation $\Integ_{\le25} =\setspec{x\in\Integ}{x \le 25}$):

\begin{paragItalic}  \label {BFGhavedegatleast21}
Let $A=\kk2$, then 
$$
\setspec{ \deg_A(F) }{ \text{$F$ is a bad field generator of $A$} } \cap \Integ_{\le25} = \{21, 25\}.
$$
\end{paragItalic}

\begin{example} \label {872349726edh83uw}
We give an infinite family of bad field generators of degree $21$.
Given $(a_0,a_1,a_2) \in \bk^3$, let $\phi(T)= T^3+a_2T^2+a_1T+a_0 \in \bk[T]=\kk1$,
and define $F_R(X,Y) \in A = \bk[X,Y]$ by
\begin{multline} \label {ydg17gmc4c7tydh}
F_R(X,Y)=X^3\phi(XY)^3+X^2\phi(XY)^2(4XY+a_2+2)\\
+ X\phi(XY)(6X^2Y^2+(4+3a_2)XY+1+a_1+a_2) \\
+4X^3Y^3+(2+3a_2)X^2Y^2+(a_2+2a_1)XY+Y.
\end{multline}
Let us sketch the proof of:
\begin{equation} \label {fjp9219020idejd0}
\begin{minipage}[t]{.9\textwidth}
\it
$F_R$ is a bad field generator of $A$ of degree $21$ and $\Delta(F_R,A) = [2,3,\dots,3]$,
where the number of occurrences of ``\,$3$'' is equal to the number of distinct roots of $\phi$.
\end{minipage}
\end{equation}
For $t \in \bk$, consider the standard homogenization
$F_R(X,Y,Z)-tZ^{21}$ of $F_R(X,Y) - t$ and the corresponding 
local equations $F_R(X,1,Z)-tZ^{21}$ and $F_R(1,Y,Z)-tZ^{21}$ at the points
$[0:1:0]$ and $[1:0:0]$ respectively.

At the point $[0:1:0]$, after the blowups $X\to XZ$ dividing by $Z^{12}$,
$Z\to XZ$ dividing by $X^8$, $X\to XZ^2$ dividing by $Z^8$, changing $X \to X-1$,
blowup $Z \to XZ$ dividing by $X^2$, again $X\to XZ$ dividing by $Z^2$,
changing $X\to X+1$, we get a dicritical of degree $2$. 

At the point $[1:0:0]$, we apply $Y\to YZ^2$ dividing by $Z^{18}$ and
we get $k$ dicriticals of degree $3$ where $k \in \{1,2,3\}$ is the number of distinct roots of $\phi$. 

So $\Delta(F_R,A) = [2,3,\dots,3]$ where ``$3$'' occurs $k$ times.
The genus formula shows that $F_R$ is a rational polynomial and hence a field generator of $A=\bk[X,Y]$;
by \ref{82364r5gg658943k}, $F_R$ is a bad field generator of $A$, proving \eqref{fjp9219020idejd0}.

Let us declare that $F,G \in A$ are {\it equivalent\/} if there exists 
$\theta \in \Aut_\bk(A)$ such that  $\theta\big( \bk[ G ] \big) = \bk[ F ]$.
We shall now prove:\footnote{The meaning of ``almost all'' is made precise in \eqref{fj92198d0192}.}
\begin{equation} \label {188092e989oye73aw8}
\begin{minipage}[t]{.9\textwidth}
\it The family $\{ F_R \}$ contains almost all
bad field generators of degree $21$ up to equivalence,
and the members of $\{ F_R \}$ are pairwise nonequivalent.
\end{minipage}
\end{equation}

Let $V$ be the set of bad field generators $F$ of $A=\bk[X,Y]$ of degree $21$ satisfying:
\begin{enumerate}

\item[(i)] the support of $F$ with respect to $(X,Y)$ is included in $\langle (0,0), (9,0), (0,12), (9,12) \rangle$;

\item[(ii)] if we write $F = \sum_{i,j} a_{i,j}X^iY^j$ ($a_{i,j} \in \bk$) then
$$
a_{9,12}=1 \text{\ \ and\ \ } a_{9,0}=a_{0,12}=a_{0,0}=0.
$$
\end{enumerate}
See \ref{90ojhgdew33689jhhf} for ``support'' and ``$\langle \dots \rangle$''.
The set of all polynomials satisfying (i) can be identified with $\aff^{130}$, so we may
view $V$ as a subset of $\aff^{130}$.
The appendix of \cite{Rus:fg2} proves, among other things, that 
(a)~every bad field generator of $A$ of degree $21$ is equivalent to an element of $V$;
and (b)~$V$ is a locally closed subset of $\aff^{130}$, and is an irreducible algebraic variety of dimension $5$
(isomorphic to the $V_1$ given there).
We shall prove:
\begin{equation} \label {fj92198d0192}
\begin{minipage}[t]{.9\textwidth}
\it   There exists a dense Zariski-open subset $U$ of $V$ such that
each element of $U$ is equivalent to a member of $\{ F_R \}$.
\end{minipage}
\end{equation}
We begin by enlarging the family $\{ F_R \}$.
Let $W = (\bk^*)^2 \times \bk^3 \subset \aff^5$ and, for each $(b,a,a_0,a_1,a_2) \in W$, define
$F(b,a,a_0,a_1,a_2 ; X,Y) \in \bk[X,Y]$ by
\begin{multline*}
F(b,a,a_0,a_1,a_2 ; X,Y)=X^3\phi(XY)^3b^4+X^2\phi(XY)^2(4b^3XY+a_2b^3+2ab^2)\\
+ X\phi(XY)(6X^2Y^2b^2+(4ab+3a_2b^2)XY+a^2+a_1b^2+ a a_2 b) \\
+4X^3Y^3b+(2a+3a_2b)X^2Y^2+(aa_2+2a_1b)XY+Y .
\end{multline*}
Note that $F(1,1,a_0,a_1,a_2 ; X,Y)$ is the right-hand-side of Equation~\eqref{ydg17gmc4c7tydh}
and hence is a member of the family $\{ F_R \}$.
One can check that, for every $(b,a,a_0,a_1,a_2) \in W$,
$t \in \bk^*$ and $(r,s,u) \in \Integ^3$ satisfying $3s-2r+u=0$,
one has
$$
t^{-r} F(b,a,a_0,a_1,a_2 ; X, Y ) = F(b t^u, at^{u+s-r},a_0t^{3(s-r)},a_1t^{2(s-r)},a_2t^{s-r} ; t^sX, t^{-r}Y )
$$
and hence
$$
F(b,a,a_0,a_1,a_2 ; X,Y) \sim F(b t^u, at^{u+s-r},a_0t^{3(s-r)},a_1t^{2(s-r)},a_2t^{s-r} ; X,Y).
$$
Taking $(r,s,u)=(1,1,-1)$ and $t=b$ gives 
$$
F(b,a,a_0,a_1,a_2 ; X,Y) \sim F(1, a/b ,a_0, a_1, a_2 ; X,Y);
$$
then taking $(r,s,u)=(3,2,0)$ and $t=a/b$ gives 
$$
F(1, a/b ,a_0, a_1, a_2 ; X,Y) \sim 
\textstyle
F(1, 1 ,a_0(\frac ab)^{-3}, a_1(\frac ab)^{-2}, a_2(\frac ab)^{-1} ; X,Y).
$$
So, for every $(b,a,a_0,a_1,a_2) \in W$, $F(b,a,a_0,a_1,a_2 ; X,Y)$ is equivalent to a member
of $\big\{ F_R \big\}$.
In particular, $F(b,a,a_0,a_1,a_2 ; X,Y)$ is a bad field generator.
It follows that if we define $G(b,a,a_0,a_1,a_2 ; X,Y)= b^{-4} F(b,a,a_0,a_1,a_2 ; Y,X)$,
then $G(b,a,a_0,a_1,a_2 ; X,Y) \in V$. So we have the morphism of varieties
$$
\psi : W \to V, \quad 
(b,a,a_0,a_1,a_2) \mapsto G(b,a,a_0,a_1,a_2 ; X,Y)
$$
and each element of the image of $\psi$ is equivalent to a member of $\big\{ F_R \big\}$.
Direct calculation shows that $G(b,a,a_0,a_1,a_2 ; X,Y) = \sum_{ij} c_{ij}X^iY^j$ satisfies
\begin{gather*}
c_{7,9} = 4/b,\\
c_{2,2} = (2a+3a_2b)/b^4,
\quad
c_{4,5} = (4a+9a_2b)/b^3,\\
c_{1,1} = (2a_1b+aa_2)/b^4,
\quad
c_{0,1} = a_0(a^2+aa_2b+a_1b^2)/b^4 .
\end{gather*}
These equations show that if  $Q=\sum_{ij} c_{ij}X^iY^j$ is an element of $V$ such that $c_{0,1}\neq 0$ then
at most one $w \in W$ satisfies  $\psi(w) = Q$.
Since the image of $\psi$ is not included in the zero-set of $c_{0,1}$, and since $\dim(W)=\dim V$,
it follows that $\psi$ is a birational morphism.
In particular, the image of $\psi$ contains a dense Zariski-open subset $U$ of $V$.
Since we have already established that each element of the image of $\psi$ 
is equivalent to a member of $\{ F_R \}$, \eqref{fj92198d0192} is proved.
This also proves the first part of claim \eqref{188092e989oye73aw8}.
We don't know whether $\{ F_R \}$ contains {\it all\/} bad field generators of degree $21$ up to equivalence. 

The aforementioned appendix also describes the possible configurations of singularities at infinity,
for a bad field generator $F$ of $A=\kk2$ of degree $21$.  That analysis
(from the last paragraph of p.~328 to the diagram at the top of p.~330) implies that
$\Delta(F,A)$ must be one of $[2,3]$, $[2,3,3]$, $[2,3,3,3]$.
It is therefore interesting to note that these three lists are realized by the family $\{ F_R \}$.

To prove the second part of claim \eqref{188092e989oye73aw8}, consider elements $F$ and $G$ of  $\{ F_R \}$
and suppose that $F \sim G$.
We may write $F = F(1,1,a_0,a_1,a_2 ; X,Y)$ and $G = F(1,1,b_0,b_1,b_2 ; X,Y)$
with $(a_0,a_1,a_2),(b_0,b_1,b_2) \in \bk^3$.
There exists $\theta \in \Aut_\bk A$ such that $\theta\big( \bk[F] \big) = \bk[G]$;
then $\theta(F) = \alpha G + \beta$ for some $\alpha \in \bk^*$ and $\beta\in\bk$.
The fact that the supports of $F$ and $G$ are included in
$\langle (0,0), (12,0), (0,9), (12,9) \rangle \setminus \{ (0,0), (12,0), (0,9) \}$
implies that $\theta(X)=uX$ and $\theta(Y)=vY$ for some $u,v \in \bk^*$.
Then we must have $F(uX,vY)=u^{12}v^9G(X,Y)$.
Write $H(X,Y)= F(uX,vY) - u^{12}v^9G(X,Y) = \sum_{i,j} h_{i,j}X^iY^j$; each $h_{i,j}$ is a polynomial
expression in $(u,v,a_0,a_1,a_2,b_0,b_1,b_2)$ that can be computed explicitly, and we must have $h_{i,j}=0$ for
all $i,j$. Calculation gives $h_{9,7} = -4u^9v^7(u^3v^2-1)$, so $u^3v^2=1$. So there exists $r \in \bk^*$ such
that $u=r^2$ and $v=sr^{-3}$ where $s=\pm1$.
After substituting these values in the expression of $H$, we find 
$h_{11,8} = 3(ra_2-sb_2)/r^3$, so $b_2=sr a_2$.
After substituting this value, we find $h_{2,2} = 2(r-s)/r^3$, so $r=s$ and hence $u=1=v$.
It follows that $F(1,1,a_0,a_1,a_2 ; X,Y) = F(1,1,b_0,b_1,b_2 ; X,Y)$ and hence that $(a_0,a_1,a_2) = (b_0,b_1,b_2)$.
This completes the proof of \eqref{188092e989oye73aw8}.

Let us also point out that Russell's polynomial is $F(1,-1,1, 3, 3 ; Y,X)$, which is equivalent to
the member $F(1,1,-1,3,-3 ; X,Y)$ of $\{F_R\}$, i.e., the member corresponding to $\phi(T)=(T-1)^3$.
It has $\Delta(F,A) = [2,3]$ and its Newton tree is given in Figure 2.
\end{example}

\begin{figure}[ht]
 
\begin{center}

\includegraphics{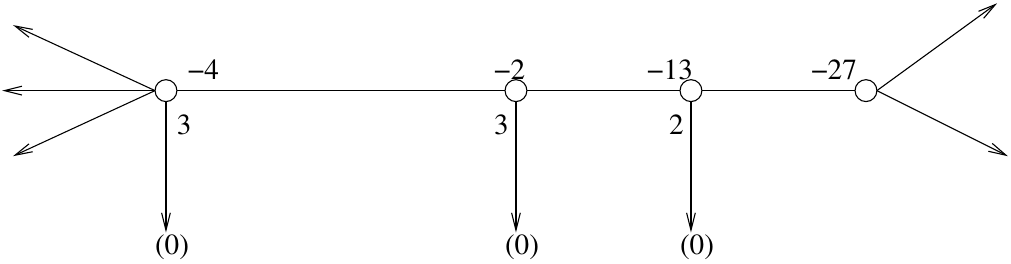}

\caption{}
\end{center}

\end{figure}

The next example gives a new family of bad field generators that shows that neither the number of dicriticals
nor their degrees are bounded (and we show in \ref{pf98287217ys82} that these bad field generators are lean).
This family generalizes Jan's polynomial.

\begin{example} \label {2983y872613765xrenucu}

Let $\phi(X) = X^n+c_{n-1}X^{n-1}+c_{n-2}X^{n-2}+\cdots+c_0 \in \bk[X]=\kk1$ where $n \ge 4$,
$c_0, \dots, c_{n-1} \in \bk$ and $c_0 \neq0$.
Denote by
$$\tilde{\phi}(X)=1+c_{n-1}X+c_{n-2}X^2+\cdots+c_0X^n$$
the reciprocal polynomial of $\phi$.
Let 
\begin{multline*}
F_{CND}(X,Y)= 
\frac{1}{\tilde{\phi}(X^{n-2}Y)}
(X+c_{n-1}X^{n-1}Y  +\cdots \\
+c_{n-i}X^{n-i}Y(X^{n-1}Y+\tilde{\phi}(X^{n-2}Y))^{i-1} 
+\cdots+c_0Y(X^{n-1}Y+\tilde{\phi}(X^{n-2}Y))^{n-1})
\end{multline*}
Thus $F_{CND} \in \bk(X,Y)$; since 
$$
X+c_{n-1}X^{n-1}Y+\cdots+c_{n-i}X^{(n-2)i+1}Y^i+\cdots+c_0X^{(n-2)n+1}Y^n=X\tilde{\phi}(X^{n-2}Y),
$$
we have $F_{CND}\in \bk[X,Y]$. This polynomial has degree $n(n-1)(n-2)+1$.
The monomial with top degree is $X^{n(n-2)^2}Y^{(n-1)^2}$. The Newton polygon has two faces. One face links the point 
$(1,0)$ to the point $(n(n-2)^2,(n-1)^2)$ and has slope $(n^2-3n+1)/(n-1)$. The other face links the point $(0,1)$ to the point $(n(n-2)^2,(n-1)^2)$ and has slope $n-2$.
We shall now prove:
\begin{equation}  \label {fo8dgs76m17w981dh912}
\begin{minipage}[t]{.9\textwidth}
\it
$F_{CND}$ is a bad field generator of $\bk[X,Y]$ with
$$
\Delta(F_{CND},\bk[X,Y])=[n-2,\cdots,n-2,n-1],
$$
where the number of dicriticals of degree $n-2$
is equal to the number of distinct roots of $\phi$.
\end{minipage}
\end{equation}

Let
$$
\Phi(X,Y,Z)=Z^{n(n-1)}+c_{n-1}X^{n-2}YZ^{(n-1)^2}+\cdots+c_0X^{n(n-2)}Y^n
$$
$$
P(X,Y,Z)=X^{n-1}YZ^{n(n-2)}+\Phi(X,Y,Z)
$$
\begin{multline*}
Q(X,Y,Z)=XZ^{n(n-1)^2}+c_{n-1}X^{n-1}YZ^{(n-1)(n^2-n-1)}+\cdots \\
+c_{n-i}X^{n-i}YP^{i-1}Z^{(n-i)(n^2-n-1)} + \cdots+c_0YP^{n-1}
\end{multline*}
$$F_t(X,Y,Z)=\frac{Q(X,Y,Z)}{\Phi(X,Y,Z)}-tZ^{n(n-1)(n-2)+1}$$

Consider first
$$
\Phi(1,Y,Z)=Z^{n(n-1)}+c_{n-1}YZ^{(n-1)^2}+\cdots+c_0Y^n
$$
$$
P(1,Y,Z)=YZ^{n(n-2)}+\Phi(1,Y,Z)
$$
\begin{multline*}
Q(1,Y,Z)=Z^{n(n-1)^2}+c_{n-1}YZ^{(n-1)(n^2-n+1)}+\cdots \\
+c_{n-i}YP^{i-1}(1,Y,Z)Z^{(n-i)(n^2-n+1)}+ \cdots+c_0YP^{n-1}(1,Y,Z)
\end{multline*}
$$
F_t(1,Y,Z)=\frac{Q(1,Y,Z)}{\Phi(1,Y,Z)}-tZ^{n(n-1)(n-2)+1}
$$
Consider the map $Y\to YZ^{n-2}$. Then:

$\Phi(1,Y,Z)=Z^{n(n-2)}\Phi_1(Y,Z)$, where  $\Phi_1(Y,Z)= Z^{n}+c_{n-1}YZ^{(n-1)}+\cdots+c_0Y^n$;

$P(1,Y,Z)=Z^{n(n-2)}P_1(Y,Z)$, where $P_1(Y,Z)= YZ^{n-2}+ \Phi_1(Y,Z)$;

$Q(1,Y,Z)=Z^{(n-2)(n^2-n+1)}Q_1(Y,Z)$, where 
\begin{multline*}
Q_1(Y,Z)= Z^{n^2-2n+2}+c_{n-1}YZ^{n^2-1}+ \cdots \\ 
+c_{n-i}YP_1^{i-1}(Y,Z)Z^{(n-i)(n+1)}+ \cdots+c_0YP_1^{n-1}(Y,Z) ;
\end{multline*}

$F_t(1,Y,Z)=Z^{(n-2)(n-1)^2}F_1(Y,Z)$, where $F_1(Y,Z)=\frac{Q_1(Y,Z)}{\Phi_1(Y,Z)}-tZ^{n^2-3n+3}$.

Consider the map $Z\to YZ$. Then:

$\Phi_1(Y,Z)=Y^n\phi(Z)$;

$P_1(Y,Z)=Y^{n-1}P_2(Y,Z)$, where $P_2(Y,Z)=Z^{n-2}+Y\phi(Z)$;

$Q_1(Y,Z)=Y^{n^2-2n+2}Q_2(Y,Z)$, where
\begin{multline*}
Q_2(Y,Z)=Z^{n^2-2n+2}+c_{n-1}Y^{2(n-1)}Z^{n^2-1}+ \cdots \\
+c_{n-i}Y^{2(n-i)}P_2^{i-1}(Y,Z)Z^{(n-i)(n+1)}+ \cdots+c_0P_2^{n-1}(Y,Z) ;
\end{multline*}

$F_1(Y,Z)=Y^{n^2-3n+2}F_2(Y,Z)$, where $F_2(Y,Z)=\frac{Q_2(Y,Z)}{\phi(Z)}-tYZ^{n^2-3n+3}$.

Finally, consider the map $Y \to YZ^{n-2}$.  Then:

$P_2(Y,Z)=Z^{n-2}P_3(Y,Z)$, where $P_3(Y,Z)=1+Y\phi(Z)$;

$Q_2(Y,Z)=Z^{(n-1)(n-2)}Q_3(Y,Z)$, where 
\begin{multline*}
Q_3(Y,Z)=Z^{n}+c_{n-1}Y^{2(n-1)}Z^{(n-1)(2n-1)}+ \cdots \\
+c_{n-i}Y^{2(n-i)}P_3^{i-1}(Y,Z)Z^{(n-i)(2n-1)}+ \cdots+c_0P_3^{n-1}(Y,Z) ;
\end{multline*}

$F_2(Y,Z)=Z^{(n-1)(n-2)}F_3(Y,Z)$, where $F_3(Y,Z)=\frac{Q_3(Y,Z)}{\phi(Z)}-tYZ^{n-1}$.

We have 
$P_3(Y,Z)= 1+c_0Y+c_1YZ+\cdots$, then $F_3(Y,0)=(1+c_0Y)^{n-1}$.
Let $Y_1=1+c_0Y$. We get $F_4(Y_1,Z)=Y_1^{n-1}-tZ^{n-1}+\sum a_{\alpha_1,\alpha_2}Y_1^{\alpha_1}Z^{\alpha_2}$ with $\alpha_1+\alpha_2\geq n-1$.
This proves that, over the point $[1:0:0]$, we get one dicritical of degree $n-1$.

Consider next
$$
\Phi(X,1,Z)=Z^{n(n-1)}+c_{n-1}X^{n-2}Z^{(n-1)^2}+\cdots+c_0X^{n(n-2)}
$$
$$
P(X,1,Z)=X^{n-1}Z^{n(n-2)}+\Phi(X,1,Z)
$$
\begin{multline*}
Q(X,1,Z)=XZ^{n(n-1)^2}+c_{n-1}X^{n-1}Z^{(n-1)(n^2-n-1)}+\cdots \\
+c_{n-i}X^{n-i}P^{i-1}Z^{(n-i)(n^2-n-1)}+ \cdots +c_0P^{n-1}
\end{multline*}
$$
F_t(X,1,Z)=\frac{Q(X,1,Z)}{\Phi(X,1,Z)}-tZ^{n(n-1)(n-2)+1}
$$
Consider the map $X\to XZ$. Then:

$\Phi(X,1,Z)=Z^{n(n-2)}\Phi^1(X,Z)$, where 
$$
\Phi^1(X,Z)= Z^{n}+c_{n-1}X^{n-2}Z^{n-1}+\cdots+c_0X^{n(n-2)};
$$

$P(X,1,Z)=Z^{n(n-2)}P^1(X,Z)$, where $P^1(X,Z)=X^{n-1}Z^{n-1}+\Phi^1(X,Z)$;

$Q(X,1,Z)=Z^{n(n-1)(n-2)}Q^1(X,Z)$, where
\begin{multline*}
Q^1(X,Z)= XZ^{n^2-n+1}+c_{n-1}X^{n-1}Z^{(n-1)n}+ \cdots \\
+c_{n-i}X^{n-i}(P^1)^{i-1}Z^{(n-i)n}+ \cdots+c_0(P^1)^{n-1} ;
\end{multline*}

$F_t(X,1,Z)=Z^{n(n-2)^2}F^1(X,Z)$, where $F^1(X,Z)=\frac{Q^1(X,Z)}{\Phi^1(X,Z)}-tZ^{n(n-2)+1}$.

Consider the map $Z \to ZX^{n-3}$. Then:

$\Phi^1(X,Z)=X^{n(n-3)}\Phi^2(X,Z)$, where $\Phi^2(X,Z)= Z^{n}+c_{n-1}XZ^{n-1}+\cdots+c_0X^{n}$;

$P^1(X,Z)=X^{n(n-3)}P^2(X,Z)$, where $P^2(X,Z)=X^2Z^{n-1}+\Phi^2(X,Z)$;

$Q^1(X,Z)=X^{n(n-1)(n-3)}Q^2(X,Z)$, where
\begin{multline*}
Q^2(X,Z)= X^{n-2}Z^{n^2-n+1}+c_{n-1}X^{n-1}Z^{(n-1)n}+ \cdots \\
+c_{n-i}X^{n-i}(P^2)^{i-1}Z^{(n-i)n}+ \cdots+c_0(P^2)^{n-1} ;
\end{multline*}

$F^1(X,Z)=X^{n(n-2)(n-3)}F^2(X,Z)$, where $F^2(X,Z)=\frac{Q^2(X,Z)}{\Phi^2(X,Z)}-tX^{n-3}Z^{n(n-2)+1}$.

Consider the map $X\to XZ$.  Then:

$\Phi^2(X,Z)=Z^n\tilde{\phi}(X)$;

$P^2(X,Z)=Z^nP^3(X,Z)$, where $P^3(X,Z)=X^2Z+\tilde{\phi}(X)$;

$Q^2(X,Z)=Z^{n(n-1)}Q^3(X,Z)$, where
\begin{multline*}
Q^3(X,Z)= X^{n-2}Z^{n-1}+c_{n-1}X^{n-1}Z^{n-1}+ \cdots \\
+c_{n-i}X^{n-i}(P^3)^{i-1}Z^{n-i}+ \cdots+c_0(P^3)^{n-1} ;
\end{multline*}

$F^2(X,Z)=Z^{n(n-2)}F^3(X,Z)$, where $F^3(X,Z)=\frac{Q^3(X,Z)}{\tilde{\phi}(X)}-tX^{n-3}Z^{n-2}$.

One can write
\begin{multline*}
F^3(X,Z)=c_0\tilde{\phi}(X)^{n-2}+Z\tilde{\phi}(X)^{n-3}A_1(X)+\cdots \\
+Z^{n-3}\tilde{\phi}(X)A_{n-3}(X)-tX^{n-3}Z^{n-2}+Z^{n-1}B(X)
\end{multline*}
This implies that, over the point $[0:1:0]$,
there are $m$ dicriticals of degree $n-2$ where $m$ is the number of distinct roots of $\phi$.
The Newton tree of $F_{CND}$ is in Figure 3. From that, or from the genus formula,
we see that $F_{CND}$ is a rational polynomial,
hence a field generator of $\bk[X,Y]$.
So assertion~\eqref{fo8dgs76m17w981dh912} is proved.
\begin{figure}[ht]
\begin{center}
\includegraphics{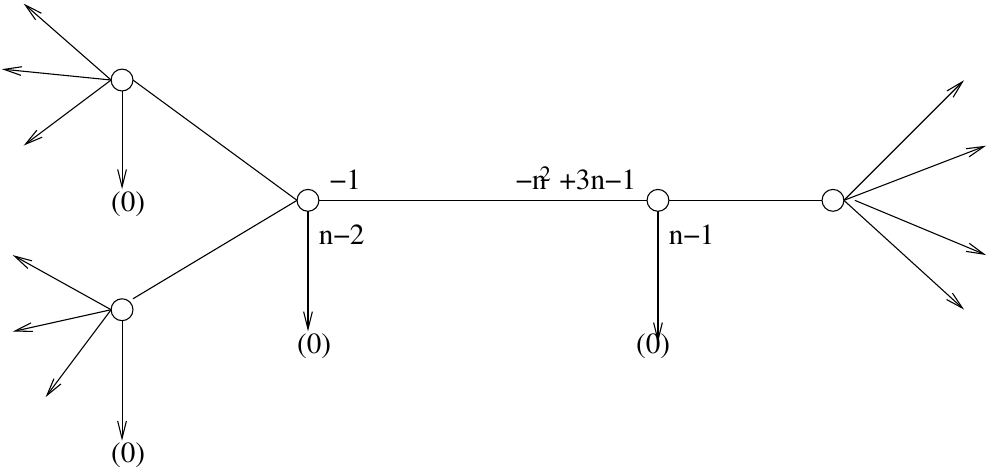}
\caption{}
\label {78r273wegadskjh278}
\end{center}
\end{figure}  
\end{example}

Until the end of this section, $\bk$ is an arbitrary field.

\begin{notation}  \label {f236te5rddsg}
Let $A = \kk2$.  Given $F \in A \setminus \bk$, we let
$\Galg(F,A)$ denote the set of prime ideals $\pgoth$ of $A$ such that
the composite $\bk[F] \hookrightarrow A \to A/\pgoth$ is an isomorphism.
We also let 
$\Gamma(F,A) = \setspec{ V(\pgoth) }{ \pgoth \in \Galg(F,A) }$,
i.e., $\Gamma(F,A)$ is the set of curves $C \subset \Spec A$ which have the property
that the composite $C \hookrightarrow \Spec A \to \Spec\bk[F]$ is an isomorphism.
Note that $\pgoth \mapsto V(\pgoth)$ is a bijection  $\Galg(F,A) \to \Gamma(F,A)$.
\end{notation}

We shall now study how field generators behave under birational extensions.

\begin{definition}  \label {xvscgdfhh82927475}
Let $\Phi : X \to Y$ be a morphism of nonsingular algebraic surfaces over $\bk$.
Assume that $\Phi$ is {\it birational}, i.e., that there exist nonempty Zariski-open subsets $U \subseteq X$
and $V \subseteq Y$ such that $\Phi$ restricts to an isomorphism $U\to V$.
By a {\it missing curve\/} of $\Phi$ we mean a curve $C \subset Y$
such that $C \cap \Phi(X)$ is a finite set of closed points.
A {\it contracting curve\/} of $\Phi$ is a curve $C \subset X$ such that $\Phi(C)$ is a point.
We write $\Miss(\Phi)$ and $\Cont(\Phi)$ for the sets of missing curves and contracting curves of $\Phi$, respectively.
Note that $\Miss(\Phi)$ and $\Cont(\Phi)$ are finite sets.
\end{definition}

\begin{notation}  \label {293832476hf}
Consider morphisms $\aff^2 \xrightarrow{ \Phi } \aff^2 \xrightarrow{ f } \aff^1$
where $\Phi$ is birational and $f$ is dominant.
Then we write
\begin{align*}
\Miss_{\text{\rm hor}}(\Phi,f)  & = \setspec{ C \in \Miss(\Phi) }{ \text{$f(C)$ is a dense subset of $\aff^1$} }.
\end{align*}
We refer to the elements of $\Miss_{\text{\rm hor}}(\Phi,f)$
as the ``$f$-horizontal'' missing curves of $\Phi$.
\end{notation}

Our next objective is to study how $\Delta(F,A)$ and $\Gamma(F,A)$ behave under a birational extension of $A$.
This is accomplished by \ref{difja;skdfj;aksd} and \ref{56fewf8r8t34kd223lwgfuio},
which are respectively results 2.9 and 3.11 of \cite{CassouDaigVGVBFG}.
See the introduction for the notation ``$A \preceq B$''.

\begin{proposition} \label {difja;skdfj;aksd}
Let $A \preceq B$ and $F \in A \setminus \bk$,
and consider the morphisms
\begin{equation*} \label  {8231h4g5hfhhcncbbxm}
\Spec B \xrightarrow{\Phi} \Spec A \xrightarrow{f} \Spec \bk[F]
\end{equation*}
determined by the inclusions $\bk[F] \hookrightarrow A  \hookrightarrow B$.
Let $C_1, \dots, C_h$ be the distinct elements of  $\Miss_{\text{\rm hor}}(\Phi,f)$ and,
for each $i \in \{ 1, \dots, h \}$,
let $\delta_i$ be the degree\footnote%
{Let $R \subseteq S$ be integral domains and $f : \Spec S \to \Spec R$ the corresponding morphism of schemes.
Assume that $\Frac S$ is a finite extension of $\Frac R$.  Then we define $\deg f = [\Frac S : \Frac R]$.}
of the morphism $f|_{C_i} : C_i \to \Spec\bk[F]$.
\begin{enumerata}

\item $\Delta(F,B) = \big[ \Delta(F,A), \delta_1, \dots, \delta_h \big]$,
i.e., $\Delta(F,B)$ is the concatenation of $\Delta(F,A)$ and $[\delta_1, \dots, \delta_h ]$.
In particular, $\dic(F,B) = \dic(F,A) + |\Miss_{\text{\rm hor}}(\Phi,f)|$.

\item For each $i \in \{1, \dots, h\}$, $\delta_i=1 \Leftrightarrow C_i \in \Gamma(F,A)$.

\end{enumerata}
\end{proposition}

\begin{lemma} 
  \label {56fewf8r8t34kd223lwgfuio}
Let $A \preceq B$ and $F \in A \setminus \bk$,
and consider the morphisms
\begin{equation*} 
\Spec B \xrightarrow{\Phi} \Spec A \xrightarrow{f} \Spec \bk[F]
\end{equation*}
determined by the inclusions $\bk[F] \hookrightarrow A  \hookrightarrow B$.
Then:
\begin{enumerata}

\item For each $C' \in \Gamma(F,B)$, $\Phi(C') \in \Gamma(F,A)$.

\item The set map $\gamma : \Gamma(F,B) \to \Gamma(F,A)$, $C' \mapsto \Phi(C')$, is injective,
and its image is the set of $C \in \Gamma(F,A)$
for which there exists
a curve $C' \subset \Spec B$ such that $\Phi|_{C'}$ is an isomorphism $C' \to C$.

\end{enumerata}
\end{lemma}

\section{The cardinality of $\Gamma(F,A)$ for field generators}
\label {SEC:CardGammaFG}

Proposition~\ref{difja;skdfj;aksd} shows the importance of $\Gamma(F,A)$ for field generators.
We will see two important features of that set: if $F$ is a field generator of $A=\kk2$ then
(i) for a suitable $(X,Y) \in \Cgoth(A)$, all elements of $\Gamma(F,A)$ are lines ``$X=$ constant'' or ``$Y=$ constant'';
(ii) except for a very special case, the cardinality of $\Gamma(F,A)$ is at most $2$.

\begin{definition} \label {90ojhgdew33689jhhf}
Let $A = \kk2$.
\begin{enumerate}

\item Given $F \in A$ and a pair $\gamma = (X,Y)$ such that $A = \bk[X,Y]$,
write $F = \sum_{i,j} a_{ij} X^iY^j$ where $a_{ij} \in \bk$ for all $i,j$;
then $\supp_\gamma(F) = \setspec{(i,j) \in \Nat^2}{ a_{ij} \neq 0 }$
is called the {\it support of $F$ with respect to $\gamma$}.

\item Given a subset $S$ of $\Reals^2$, let $\langle S \rangle$ denote its convex hull.

\item \label {3rfhbh8c98b82yh45ft6}
Given $F \in A$, we write $\Rec(F,A)$ for the set of ordered pairs $\gamma = (X,Y)$ satisfying 
$A = \bk[X,Y]$ and
$$
\text{there exist $m,n \ge 1$ such that 
$(m,n) \in \supp_\gamma(F) \subseteq \langle (0,0),(m,0),(0,n),(m,n) \rangle$.}
$$
Let $\Rec^+(F,A)$ be the set of $\gamma=(X,Y) \in \Rec(F,A)$ satisfying
the additional condition ``$m\le n$''.
Clearly, 
$$
\Rec^+(F,A) \neq \emptyset \Leftrightarrow \Rec(F,A) \neq \emptyset.
$$

\item By a {\it rectangular element of $A$} we mean an $F\in A$ satisfying $\Rec(F,A) \neq \emptyset$.

\end{enumerate}
\end{definition}

See \ref{982393298urcnj092}, below,
to understand why the notion of rectangular element is relevant for studying field generators.

\begin{lemma}[{\cite[3.4]{CassouDaigVGVBFG}}]  \label {0cv9v8n8dllrklslk52kf0}
Let $F$ be a rectangular element of $A = \kk2$.
\begin{enumerata}

\item If $(X,Y) \in \Rec(F,A)$ then 
\begin{multline*}
\Rec(F,A) = \setspec{ (aX+b, cY+d) }{ a,b,c,d \in \bk,\ ac\neq0 } \\
\cup \setspec{ (cY+d, aX+b) }{ a,b,c,d \in \bk,\ ac\neq0 }.
\end{multline*}

\item \label {9834fh87td4rth}
Up to order, the pair $(m,n)$ in \ref{90ojhgdew33689jhhf}\eqref{3rfhbh8c98b82yh45ft6} depends only on $(F,A)$,
i.e., is independent of the choice of $\gamma \in \Rec(F,A)$.

\end{enumerata}
\end{lemma}

\begin{definition}  \label {89798xchh0c09x9cww9}
For each rectangular element $F$ of $A = \kk2$ we define 
$$
\text{$\bideg_A(F) = (\deg_X(F), \deg_Y(F))$ for any $(X,Y) \in \Rec^+(F,A)$.}
$$
By \ref{0cv9v8n8dllrklslk52kf0}, $\bideg_A(F)$ is well defined and depends only on $(F,A)$.\\
Moreover, if $(m,n) = \bideg_A(F)$ and $\gamma \in \Rec^+(F,A)$ then
$$
\text{$1 \le m \le n$\ \ and\ \  $(m,n) \in \supp_\gamma(F) \subseteq \langle (0,0),(m,0),(0,n),(m,n) \rangle$.}
$$
\end{definition}

\begin{remark}
Let $F$ be a rectangular element of $A = \kk2$ and let $(m,n) = \bideg_A(F)$.
It follows from \ref{0cv9v8n8dllrklslk52kf0}\eqref{9834fh87td4rth} that if $m=n$ then
$\Rec^+(F,A) = \Rec(F,A)$.
\end{remark}

We shall now consider the set $\Galg(F,A)$ defined in \ref{f236te5rddsg}.
By the next fact, $\Galg(F,A)$ is  easy to describe when $F$ is a rectangular element of~$A$.

\begin{lemma}[{\cite[3.7]{CassouDaigVGVBFG}}]  \label {09239023r02n02b27c2c8h2}
Let $F$ be a rectangular element of $A = \kk2$, $\gamma = (X,Y) \in \Rec(F,A)$ and
$(m,n) = ( \deg_X(F), \deg_Y(F) )$. Recall that
$$
(m,n) \in \supp_\gamma(F) \subseteq \langle (0,0),(m,0),(0,n),(m,n) \rangle.
$$
Write $F = \sum_{i,j} a_{ij} X^iY^j$ ($a_{ij} \in \bk$) and define
$$
\textstyle
F_\text{\rm ver}(Y) = \sum_{j=0}^n a_{m,j}Y^j \quad \text{and} \quad
F_\text{\rm hor}(X) = \sum_{i=0}^m a_{i,n}X^i .
$$
\begin{enumerata}

\item \label {09293bbgxqweiby7sznxadfefj}
$\Galg(F,A)$ is equal to
$$
\setspec{ (X-a) }{ a\in\bk \text{ and } \deg F(a,Y)=1 }
\cup \setspec{ (Y-b) }{ b\in\bk \text{ and } \deg F(X,b)=1 }.
$$

\item \label {09dsfjk3346gbxbcdzrew24} If $\min(m,n)>1$ then $\Galg(F,A)$ is included in
$$
\setspec{ (X-a) }{ a\in\bk \text{ and } F_\text{\rm hor}(a)=0 }
\cup \setspec{ (Y-b) }{ b\in\bk \text{ and } F_\text{\rm ver}(b)=0 }.
$$

\end{enumerata}
\end{lemma}

The next result (\ref{982393298urcnj092}) is due to Russell,
and has proved to be very useful in the study of field generators.
Here, one should observe that no variable of $A=\kk2$ is a rectangular element of $A$,
because any rectangular element has two points at infinity. 

\begin{theorem}[{\cite[3.7 and 4.5]{Rus:fg}}]  \label {982393298urcnj092}
If $F$ is a field generator of $A=\kk2$ which is not a variable of $A$, then $F$ is a rectangular element of $A$.
\end{theorem}

We now turn our attention to the cardinality of $\Gamma(F,A)$ where $F$ is a field generator of $A=\kk2$.
Note that there is no upper bound on $| \Gamma(F,A) |$
for rectangular elements $F$ of $A$ and even for certain types of field generators:

\begin{examples}  \label {8923ry7fhuh}
Let $A = \bk[X,Y]=\kk2$.
\begin{enumerata}

\item Let $F = u(Y)X^2 + X$, where $\deg u(Y) > 1$;
then $F$ is a rectangular element of $A$ and \ref{09239023r02n02b27c2c8h2} implies that $|\Gamma(F,A)|$ equals
the number of roots of $u(Y)$.

\item \label {923hvmnbwkuwq5}
Assume that $\bk$ is infinite and let $F = \alpha(Y)X + \beta(Y)$
where $\deg\beta(Y) \le \deg \alpha(Y) > 0$;
then \ref{09239023r02n02b27c2c8h2} implies that $|\Gamma(F,A)| = |\bk|$.
Since $\bk(F,Y)=\bk(X,Y)$,
$F$ is a good field generator of $A$ (of an especially simple type).

\end{enumerata}
\end{examples}

\begin{theorem}[{\cite[4.11]{CassouDaigVGVBFG}}]  \label {9854dnc2mrhvfc}
Let $F$ be a field generator of $A=\kk2$ satisfying $| \Gamma(F,A) | > 2$.
Then there exists $(X,Y)$ such that $A = \bk[X,Y]$ and $F = \alpha(Y)X+\beta(Y)$
for some $\alpha(Y), \beta(Y) \in \bk[Y]$.
\end{theorem}

The above theorem is one of the main results of \cite{CassouDaigVGVBFG}.
Its corollary (below) has interesting consequences in the classification of field generators
(for instance, \ref{3q939xuh23829x9mh812o}\eqref{x23c3x5v4c5bv6nbm7nki0p} is needed
in the proof of \ref{q23328r83yd74r9128}\eqref{b-734tc8rgj3rfgh}).

\begin{corollary}[{\cite[4.12]{CassouDaigVGVBFG}}]    \label {3q939xuh23829x9mh812o} 
If  $F$ is a bad field generator of $A = \kk2$ then the following hold.
\begin{enumerata}

\item \label {x23c3x5v4c5bv6nbm7nki0p} $| \Gamma(F,A) | \le 2$ 

\item \label {xvbvn5v5n7bm8m68bn43v53c}
$\Rec(F,A) \neq \emptyset$ and the pair $(m,n) = \bideg_A(F)$ satisfies $2 \le m \le n$.

\item \label {c-84397bcfdv563} There exists $(X,Y) \in \Rec(F,A)$ such that
$$
\Galg(F,A) \subseteq \{ (X), (Y) \} \quad \text{or} \quad \Galg(F,A) \subseteq \{ (X), (X-1) \}.
$$

\end{enumerata}
\end{corollary}

\begin{examples}  \label {98f298e91us98b1xo27182}
Using \ref{09239023r02n02b27c2c8h2}, we see that
$\Galg(F_J,A) = \{ (X), (Y) \}=\Galg(F_{CND},A)$ and  $\Galg(F_R,A) = \{ (X) \}$.
We will see examples  of bad field generators $F$ of $A$ satisfying $\Galg(F,A) =\emptyset$,
but  we do not have examples such that $ \Galg(F,A) = \{ (X), (X-1) \}.$
\end{examples}

\section{Very good and very bad field generators}
\label {SEC:VGVBFGs}

The next proposition gives a partial characterization of very good field generators.
For the moment, this is the best that we can say on that subject.
(In part \eqref{a-734tc8rgj3rfgh}, let us agree that $\gcd\emptyset=\infty$.)

\begin{proposition}[{\cite[5.3]{CassouDaigVGVBFG}}]     \label {q23328r83yd74r9128}
Let $F$ be a field generator of $A=\kk2$ and $\Delta(F,A) = [d_1, \dots, d_t]$.
\begin{enumerata}

\item \label {a-734tc8rgj3rfgh}
If $\gcd\big( \{ d_1, \dots, d_t \} \setminus \{1\} \big)>1$
then $F$ is a very good field generator of $A$.
In particular,
if at most one $i \in \{1, \dots, t\}$ satisfies $d_i>1$ then $F$ is a very good field generator of $A$.

\item \label {b-734tc8rgj3rfgh}
If at least three $i \in \{1, \dots, t\}$ satisfy $d_i=1$ then $F$ is a very good field generator of $A$.

\item \label {c-734tc8rgj3rfgh} If $F$ is a good but not very good field generator of $A$
then 
$$
\Delta(F,A) = [1, \dots, 1, e_1, \dots, e_s]
$$
where \mbox{\rm ``$1$''} occurs either $1$ or $2$ times,
$s \ge 2$, $\min(e_1, \dots, e_s)>1$ and\\
$\gcd(e_1, \dots, e_s)=1$.

\end{enumerata}
\end{proposition}
\begin{remark}  \label {pf9823f898012dj}
By~\ref{q23328r83yd74r9128}\eqref{a-734tc8rgj3rfgh},
the polynomials classified in \cite{MiySugie:GenRatPolys},
\cite{NeumannNorbury:simple} and \cite{Sasao_QuasiSimple2006} are special cases of very good field generators.
This gives many complicated examples of very good field generators.
\end{remark}

\begin{example} \label {87hf25f34w32a1a} 
Let $F$ be a bad field generator of $A = \bk[X,Y]$ such that 
$\Galg(F,A) = \{ (X), (Y) \}$ and $\Delta(F,A) = [3,4]$.
For example take $F_{CND}$ for $\phi(X)=(X-1)^5$.
Let $B = \bk[X/Y,\, Y^2/X]$ and note that $A \preceq B$.
Consider the morphisms
$\Spec B \xrightarrow{\Phi} \Spec A \xrightarrow{f} \Spec \bk[F]$
determined by the inclusions $\bk[F] \hookrightarrow A  \hookrightarrow B$.
Then the missing curves of $\Phi$ are $C_1=V(X)$ and $C_2=V(Y)$ and these are $f$-horizontal,
so $\Miss_{\text{\rm hor}}(\Phi,f) = \{ C_1, C_2 \}$.
In the notation of \ref{difja;skdfj;aksd} we have $\delta_1=\delta_2=1$ (because $C_1,C_2 \in \Gamma(F,A)$),
so that result implies that $\Delta(F,B) = [3,4,1,1]$.
Note that {\it $F$ is not a very good field generator of $B$} (because it is bad in $A$).
This shows that, in \ref{q23328r83yd74r9128}\eqref{b-734tc8rgj3rfgh},
one cannot replace ``at least three'' by ``at least two'';
and in the second part of \ref{q23328r83yd74r9128}\eqref{a-734tc8rgj3rfgh}, one cannot replace
``at most one'' by ``at most two''.

Another application of \ref{difja;skdfj;aksd} shows that $F$ is
a good field generator of $B'=\bk[X,Y/X]$ which is not very good and which has $\Delta(F,B') = [3,4,1]$.
\end{example}

\begin{remark}  \label {o832bi87di2837qtduj}
We shall give in \ref{26tsgfvbxnbcbhjdhbkizey} an 
example of a very good field generator $F$ of $A = \kk2$ with $\Delta(F,A)=[3,4,1]$.
So the converse of \ref{q23328r83yd74r9128}\eqref{c-734tc8rgj3rfgh} is not true.
Moreover, noting that $\Delta(F,B') = [3,4,1]$ in \ref{87hf25f34w32a1a}, this will also show that
\begin{quote}
\it There exist good field generators $F,G$ of $A = \kk2$ such that $F$ is very good,
$G$ is not very good, and $\Delta(F,A) = [3,4,1] = \Delta(G,A)$.
\end{quote}
So the degree list $\Delta(F,A)$ does not characterize very good field generators among good field generators.
\end{remark}

The set $\Galg(F,A)$ characterizes very bad field generators among bad field generators: 
result \ref{88x889adb823mdnfv}\eqref{n2n92bm2brxgftr}
gives such a characterization and, in fact,
makes it easy to decide whether a given bad field generator is very bad.

\begin{proposition}[{\cite[5.8]{CassouDaigVGVBFG}}]       \label {88x889adb823mdnfv}
Let $F$ be a bad field generator of $A = \kk2$.
\begin{enumerata}

\item \label {n2n92bm2brxgftr}
$F$ is a very bad field generator of $A$ if and only if $\Galg(F,A) = \emptyset$.

\item Suppose that $\Galg(F,A) \neq \emptyset$.
Then there exists $(X,Y)$ such that $A = \bk[X,Y]$ and $(X) \in \Galg(F,A)$.
For any such pair $(X,Y)$, $F$ is a good field generator of $\bk[X, Y/X]$.

\end{enumerata}
\end{proposition}

\begin{example} \label {GoodBadUgly}
Let $F$ be a bad field generator of $A = \bk[X,Y]=\kk2$ such that $\Galg(F,A) = \{ (X) \}$
(for instance, $F = F_R$ for any choice of $\phi$; see \ref{98f298e91us98b1xo27182}).
By \ref{88x889adb823mdnfv}, $F$ is not a very bad field generator of $A$.
Let $B = \bk[X/(Y-1), Y]$ and note that $F \in A \preceq B$; so $F$ is a field generator of $B$ and we claim:
\begin{equation}  \label {fo928b9827y3913}
\text{$F$ is a very bad field generator of $B$.}
\end{equation}
To see this, consider $\Spec B \xrightarrow{\Phi} \Spec A \xrightarrow{f} \Spec \bk[F]$ as in \ref{difja;skdfj;aksd}.
We have $\Gamma(F,A) = \{D\}$ where $D = V(X) \subset\Spec A$ and
$\Miss\Phi = \{ C \}$ 
where $C = V(Y-1) \subset \Spec A$.
If $f(C)$ is a point then $\Miss_{\text{\rm hor}}(\Phi,f) = \emptyset$ and
\ref{difja;skdfj;aksd}(a) implies that $\Delta(F,B) = \Delta(F,A)$;
if $f(C)$ is not a point then $\Miss_{\text{\rm hor}}(\Phi,f) = \{ C \}$ 
and \ref{difja;skdfj;aksd}(a) implies that $\Delta(F,B) = [\Delta(F,A), \delta ]$ where $\delta$ is
the degree of $f|_C : C \to \Spec A$,
and where $\delta>1$ by \ref{difja;skdfj;aksd}(b) and because $C \notin \Gamma(F,A)$.
Since ``$1$'' does not occur in $\Delta(F,A)$ by \ref{82364r5gg658943k},
it follows (in both cases) that it does not occur in $\Delta(F,B)$ either;
so (by \ref{82364r5gg658943k} again) $F$ is a bad field generator of $B$.

In view of \ref{88x889adb823mdnfv}, there only remains to show that $\Gamma(F,B) = \emptyset$.
Suppose that there exists an element $D'$ of $\Gamma(F,B)$.
Then \ref{56fewf8r8t34kd223lwgfuio} implies that $\Phi|_{D'} : D' \to D$ is an isomorphism.
This is not the case, because $D \nsubseteq \image\Phi$
(the maximal ideal $(X,Y) \in \Spec A$ is a point of $D$ but not of $\image\Phi$).
This proves \eqref{fo928b9827y3913}.
\end{example}

Observe that the very bad field generators (of $B$) constructed in \ref{GoodBadUgly} are not lean in $B$,
due to the method of construction.
All examples of very bad field generators given in \cite{CassouDaigVGVBFG} are constructed by that same method,
and hence are not lean.
In \ref{p92839819ueg87gdywfe6}, below, we give the first example of a very bad field generator that is also lean.

\section{Lean field generators}
\label {Leanbadfieldgenerators}

See the introduction for the statement of the problems ``(i)'' and ``(ii)'', that will occupy us in this section.
See in particular \ref{823gdvcuuyv2trc6374rnmujfj} for the definition of the property of being lean.
We immediately observe:

\begin{lemma}  \label {dg72576wqs81128753d}
If $F$ is a good field generator of $A = \kk2$, then $F$ is not lean in $A$.
\end{lemma}

\begin{proof}
Since $F$ is a good field generator of $A$, we may pick
$G\in A$ such that $\Frac A = \bk(F,G)$; then $F \in \bk[F,FG] \preceq A$
and $\bk[F,FG] \neq A$.
\end{proof}

The paper \cite{CN-Daigle:Lean} (in preparation) contains the following result:

\begin{theorem} \label {546123r5eghjx}
For $F \in A = \kk2$, the following are equivalent:
\begin{enumerata}

\item $F \in A$ does not have a lean factorization.

\item There exists a very good field generator $G$ of $A$ such that $F \in \bk[G]$.

\end{enumerata}
\end{theorem}

We shall now give the proof of the special case \ref{888dfgdzb5n3d2ff} of \ref{546123r5eghjx},
because it is considerably simpler than that of the general case,
and because we only need this special case in the present paper.
The proof of the special case is based on simple minded degree considerations,
an approach that does not seem to work in the general case.
See the introduction for the notations regarding degree. 

\begin{lemma}  \label {of82098u12bw9}
Let $F$ be a rectangular element of $A = \kk2$.
\begin{enumerata}

\item $\deg_A(F) = \deg_\rho(F)$ for any $\rho \in \Rec(F,A)$.

\item $\deg_A(F) < \deg_B(F)$ for every strict inclusion $A \subset B = \kk2$.

\end{enumerata}
\end{lemma}

\begin{proof}
Consider an inclusion $A \subseteq B$ where $A=\kk2$, $B=\kk2$ and $F$ is a rectangular element of $A$.
Choose $\gamma = (X,Y) \in \Cgoth(B)$ such that $\deg_\gamma(F) = \deg_B(F)$.
Choose $\rho = (U,V) \in \Rec(F,A)$; then
$$
\text{there exist $m,n \ge 1$ such that 
$(m,n) \in \supp_\rho(F) \subseteq \langle (0,0),(m,0),(0,n),(m,n) \rangle$.}
$$
If we define $e_1 = \deg_\gamma(U)$ and $e_2 = \deg_\gamma(V)$ then
$\deg_\gamma(U^mV^n) = me_1+ne_2$ and, for each $(i,j) \in \supp_\rho(F) \setminus \{(m,n)\}$,
$\deg_\gamma(U^iV^i) = ie_1+je_2 <  me_1+ne_2$; so 
$$
\deg_B(F) = \deg_\gamma F = me_1+ne_2 .
$$ 
Since $\deg_\rho(F) = m+n$, we have
\begin{equation} \label {z9198287g1s9wq}
\deg_B(F) - \deg_\rho(F) \ge m(e_1-1) + n(e_2-1) \ge 0
\end{equation}
and consequently
\begin{equation} \label {92gd61dyvdkwo}
\deg_A(F) \le \deg_\rho(F) \le \deg_B(F).
\end{equation}
Assertion (a) follows from the special case $A=B$ of \eqref{92gd61dyvdkwo}.
To prove (b) we note that
the condition $\deg_A(F) = \deg_B(F)$ implies  (by (a) and \eqref{z9198287g1s9wq}) that $e_1=1=e_2$,
so $A = \bk[U,V] = \bk[X,Y] = B$.
\end{proof}

\begin{proposition} \label {888dfgdzb5n3d2ff}
For a field generator $F$ of $A = \kk2$, the following are equivalent:
\begin{enumerata}

\item $F \in A$ has a lean factorization.

\item $F$ is not a very good field generator of $A$.

\end{enumerata}
\end{proposition}

\begin{proof}
Suppose that (b) holds. Then there exists $A'$ such that
$F \in A' \preceq A$ and $F$ is a bad field generator of $A'$.
Consider the set $\Sigma = \setspec{ R }{ \text{$F \in R \preceq A'$}}$, which is nonempty since $A' \in \Sigma$.
For each $R \in \Sigma$, $F$ is a bad field generator of $R$ and hence (by \ref{982393298urcnj092})
a rectangular element of $R$.
So \ref{of82098u12bw9} implies that if $R_1 \subset R_2$ is a strict inclusion with $R_1,R_2 \in \Sigma$,
then $\deg_{R_1}(F) < \deg_{R_2}(F)$. It follows that $\Sigma$ has a minimal element $R_0$.
Then $F \in R_0 \preceq A$ and $F$ is lean in $R_0$, so (a) is true.

Conversely, suppose that (a) holds.
Then there exists $A'$ such that $F\in A'\preceq A$ and $F$ is lean in $A'$. 
Then (by \ref{dg72576wqs81128753d}) $F$ is a bad field generator of $A'$, so (b) holds.
\end{proof}

Result \ref{888dfgdzb5n3d2ff} partially solves problem ``(i)'' stated in the introduction.
To complete the solution of (i) there would remain to classify very good field generators,
but this question is open.

We shall now make some modest contributions to the problem  (called ``(ii)'' in the introduction)
of classifying lean field generators.
Recall (\ref{dg72576wqs81128753d}) that if a field generator is lean then it is bad.
Also observe that,
by \ref{888dfgdzb5n3d2ff}, it is a priori clear that lean field generators exist
(we know that there exists a bad field generator $F$ of $A=\kk2$,
and \ref{888dfgdzb5n3d2ff} implies that there exists $A'$ satisfying $F \in A' \preceq A$ and such that
$F$ is lean in $A'$; then $F$ is a bad field generator of $A'$ that is lean in $A'$).
For specific examples, consider \ref{872349726edh83uw} with:

\begin{corollary}
Every bad field generator of degree $21$ is lean.
\end{corollary}

\begin{proof}
Let $F$ be a bad field generator of $B = \kk2$ such that $\deg_B(F)=21$.
By contradiction, assume that $F$ is not lean in $B$.
Then $F \in A \preceq B$ for some $A$ such that $A \neq B$.
Then $F$ is a bad field generator of $A$, so \ref{BFGhavedegatleast21} gives the first inequality in
$$
21 \le \deg_A(F) < \deg_B(F)
$$
while the second inequality is \ref{of82098u12bw9}.
This contradicts the assumption.
\end{proof}

It is much more difficult to determine whether there exist {\it very\/} bad field generators that are lean.
In fact this question remained open for several years.
The very bad field generators exhibited in \ref{GoodBadUgly} are---by construction---not lean.
In \ref{p92839819ueg87gdywfe6}, below, we give an example of a very bad field generator that is lean.
First, we need to develop some tools.

\begin{notation}
Given a birational morphism $\Phi : \aff^2 \to \aff^2$ 
and a coordinate system $\gamma$ of $\aff^2$, we write
$$
\delta_\gamma( \Phi ) = \begin{cases}
\max\setspec{ \deg_\gamma(C) }{ C \in \Cont(\Phi) }, 
& \text{if $\Cont(\Phi) \neq \emptyset$;} \\ 
1, & \text{if $\Cont(\Phi) = \emptyset$}.
\end{cases}
$$
See \ref{xvscgdfhh82927475} for the notation $\Cont(\Phi)$, and observe\footnote{This is well known 
and easy to show when $\bk$ is algebraically closed (e.g., \cite[2.6(b)]{CassouDaigBir}),
and it is straightforward to generalize this to arbitrary $\bk$.}
that $\Cont(\Phi)$ is empty if and only if $\Phi$ is an automorphism of $\aff^2$.
\end{notation}

\begin{lemma}  \label {m93919gdqwd8719}
Assume that $\bk$ is algebraically closed.
Let $F \in A \preceq B$ where $F$ is a rectangular element of $A$ and
let $\Phi : \Spec B \to \Spec A$ be the morphism determined by the inclusion $A \to B$.
Then
$$
\deg_A(F) \le \frac{ \deg_\gamma(F) }{ \delta_\gamma(\Phi) }
\quad \text{for every $\gamma \in \Cgoth(B)$.}
$$
Moreover, if $A \neq B$ then the above inequality is strict for all $\gamma \in \Cgoth(B)$.
\end{lemma}

\begin{proof}
Let $\gamma = (X,Y) \in \Cgoth(B)$.
If $\Cont(\Phi) = \emptyset$ then $\delta_\gamma(\Phi) = 1$ and the claim
is an immediate consequence of \ref{of82098u12bw9}.
So we may assume throughout that $\Cont(\Phi) \neq \emptyset$.
Under this assumption we shall prove that $\deg_A(F) < \deg_\gamma(F) / \delta_\gamma(\Phi)$.

Choose $C \in \Cont(\Phi)$ such that $\delta_\gamma(\Phi) = \deg_\gamma(C)$;
let $P$ be an irreducible element of $B$ such that $C = V(P)$.
Pick $\rho = (U,V) \in \Rec(F,A)$. 
Let us first prove that
\begin{equation}  \label {op929081271ys}
\min ( e_1 , e_2 ) \ge  \delta_\gamma(\Phi)
\end{equation}
where we define $e_1 =  \deg_\gamma U$ and $e_2 = \deg_\gamma V$.
Note that $\Phi$ is given by the formula $\Phi(x,y) = (U(x,y), V(x,y))$,
where we use coordinates $X,Y$ (resp.\ $U,V$) to identify the set of closed points of  $\Spec B$ (resp.\ $\Spec A$)
with $\bk^2$.
As $\Phi$ maps $V(P)$ to a point $(a,b)$, we have $P \mid \gcd_B(U-a,V-b)$.
In particular, 
\begin{equation} \label {pf2q9897ygsid712}
\min ( e_1 , e_2 ) = \min ( \deg_\gamma U, \deg_\gamma V ) \ge \deg_\gamma P = \deg_\gamma C = \delta_\gamma(\Phi) .
\end{equation}
So \eqref{op929081271ys} is proved.
Since $\rho \in \Rec(F,A)$, 
$$
\text{there exist $m,n \ge 1$ such that 
$(m,n) \in \supp_\rho(F) \subseteq \langle (0,0),(m,0),(0,n),(m,n) \rangle$.}
$$
We have $\deg_\gamma(U^mV^n) = me_1+ne_2$ and, for each $(i,j) \in \supp_\rho(F) \setminus \{(m,n)\}$,
$\deg_\gamma(U^iV^i) = ie_1+je_2 <  me_1+ne_2$; so 
\begin{equation} \label {5e43qwfgsdi09awe}
\deg_\gamma F = me_1+ne_2 \ge (m+n) \min(e_1,e_2) .
\end{equation}
Since $m+n = \deg_\rho(F) = \deg_A(F)$ by \ref{of82098u12bw9}, inequalities \eqref{op929081271ys}
and \eqref{5e43qwfgsdi09awe} imply 
\begin{equation} \label {8372yhs8d9a}
\deg_A(F) \le \frac{ \deg_\gamma(F) }{ \delta_\gamma(\Phi) } .
\end{equation}
Assume for a moment that equality holds in \eqref{8372yhs8d9a}.
Then equality must hold in \eqref{5e43qwfgsdi09awe} (so $e_1=e_2$) and in \eqref{pf2q9897ygsid712}
(so $\deg_\gamma P = \deg_\gamma U = \deg_\gamma V$). As  $P \mid \gcd_B(U-a,V-b)$, it follows that 
$P, U-a, V-b$ are associates, so $U,V$ are algebraically dependent, a contradiction.
So inequality \eqref{8372yhs8d9a} is strict, and the lemma is proved.
\end{proof}

\begin{notation}
Let $B=\kk2$ and $F\in B \setminus\bk$. Write $\Delta (F,B)=[d_1,\cdots,d_t]$ and
let $\Deul(F,B)$ be the set of nonempty subsets $I$ of $\{1, \dots, t\}$ satisfying
$\forall_{i \in I}( d_i \neq 1 )$ and $\gcd\setspec{ d_i }{ i \in I } = 1$.
Define:
$$
\delta(F,B) = \begin{cases}
\min\setspec{ \textstyle\sum_{i \in I}d_i }{ I \in \Deul(F,B) }, & \text{if $\Deul(F,B) \neq \emptyset$}; \\
\infty,  & \text{if $\Deul(F,B) = \emptyset$}.
\end{cases}
$$
\end{notation}

\begin{lemma}  \label {23h7232sqwadzcpydgb}
Let $F$ be a field generator of $B = \kk2$.
Then
$$
\delta(F,B) \le \deg_A(F) \le \deg_B(F) 
$$
for every $A$ satisfying $F \in A \preceq B$ and such that $F$ is a bad field generator of $A$.
\end{lemma}

\begin{proof}
Suppose that $F \in A \preceq B$ and  that $F$ is a bad field generator of $A$.
We have $\deg_A(F) \le \deg_B(F)$ by \ref{of82098u12bw9}.
By \ref{difja;skdfj;aksd}, $\Delta(F,A)$ is a sublist of $\Delta(F,B)$; so we may write
$\Delta(F,A) = [d_1, \dots, d_s]$
and $\Delta(F,B) = [d_1, \dots, d_t]$ where $s\le t$.
Since $F$ is a bad field generator of $A$, we have
$\forall_{i \in \{1,\dots,s\} }( d_i \neq 1 )$ by \ref{82364r5gg658943k}
and $\gcd(d_1, \dots, d_s)=1$ by \ref{2384y872tqo8we1ftbs},
so $\{1, \dots, s\} \in \Deul(F,B)$ and hence $\sum_{i=1}^s d_i \ge \delta(F,B)$.
Since $\deg_A(F) \ge \sum_{i=1}^s d_i$ by \ref{828dbvcdqscd51671w72}, we are done.
\end{proof}

\begin{notation}  \label {o273e27367wey128}
Let $F \in B \setminus \bk$, where $B = \kk2$. For each $\gamma \in \Cgoth(B)$ we define
$$
\subdeg_\gamma(F) = \min\setspec{ \deg_\gamma(P) }{ P \in \Peul }
$$
where $\Peul$ denotes the set of irreducible elements $P$ of $B$ such that
$P \mid (F-\lambda)$ in $B$ for some $\lambda \in\bk$.
\end{notation}

\begin{corollary} \label {last}
Assume that $\bk$ is algebraically closed and
let $F$ be a field generator of $B = \kk2$.
If there exists $\gamma \in \Cgoth(B)$ satisfying
$$
\delta(F,B)\ \ge\ \frac{\deg_\gamma(F)}{\subdeg_\gamma(F)}
$$
then exactly one of the following conditions is satisfied:
\begin{enumerata}

\item \label {BFGL} $F$ is a bad field generator of $B$ and is lean in $B$;

\item  \label {VG000000} $F$ is a very good field generator of $B$.

\end{enumerata}
\end{corollary}

\begin{proof}
Assume that there exists a $\gamma$ as in the above statement and that condition \eqref{VG000000} is not satisfied.
Then $E \neq \emptyset$, where we define
$$
E = \setspec{ A }{ \text{$F \in A \preceq B$ and $F$ is a bad field generator of $A$ }}.
$$
We claim that $E = \{B\}$.  Indeed, assume the contrary.
Because $E \neq \emptyset$, there exists an $A \in E$ satisfying $A \neq B$.
Then the birational morphism $\Phi : \Spec B \to \Spec A$ is not an isomorphism;
so (cf.\ \cite[2.6(b)]{CassouDaigBir}) $\Cont(\Phi) \neq \emptyset$ and each element of $\Cont(\Phi)$ is 
an irreducible component of a fiber of $\Spec B \to \Spec \bk[F]$;
consequently, $\delta_\gamma(\Phi) \ge \subdeg_\gamma(F)$.
By the assumption, it follows that $\delta(F,B) \ge \deg_\gamma(F)/ \delta_\gamma(\Phi)$.
Note that $F$ is a rectangular element of $A$, because it is a bad field generator of $A$.
So by \ref{m93919gdqwd8719} we get $\deg_\gamma(F)/ \delta_\gamma(\Phi) > \deg_A(F)$,
so $\deg_A(F) < \delta(F,B)$,  which contradicts \ref{23h7232sqwadzcpydgb}.

This contradiction shows that $E=\{B\}$ and hence that \eqref{BFGL} holds.
\end{proof}

The following is a technical lemma that we need in \ref{26tsgfvbxnbcbhjdhbkizey},
\ref{p92839819ueg87gdywfe6} and \ref{pf98287217ys82}.
We say that an element $r$ of a ring $R$ ``is a power'' if $r = r_0^k$ for some
$r_0 \in R$ and $k>1$.

\begin{lemma}  \label {783e826gehfood0293}
Assume that $\bk$ is algebraically closed and that $F$ is a rational polynomial of $A = \kk2$.
Let $d=\dic(F,A)$, let $t_1, \dots, t_{d-1}$ be distinct elements of $\bk$
and suppose that, for each $i \in \{1, \dots, d-1\}$, we have a factorization
$F-t_i = G_iH_i$ with $G_i,H_i \in A \setminus \bk$.
Moreover, assume that for each $i \in \{1, \dots, d-1\}$
there exists $(X,Y) \in \Cgoth(A)$ such that
one of the following conditions is satisfied:

\begin{enumerate}

\item[(i)]
$G_i(X,0)\notin \bk$ and $G_i(X,0)$ is not a power,
$H_i(X,0)\notin \bk$ and $H_i(X,0)$ is not a power,
and $\gcd (G_i(X,0),H_i(X,0))=1$;

\item[(ii)]
$G_i(X,0)\in\bk^*$, $H_i(X,0)\notin \bk$ and $H_i(X,0)$ is not a power,
$G_i(0,Y)\notin \bk$ and $G_i(0,Y)$ is not a power, and $H_i(0,Y)\in\bk^*$.

\end{enumerate}
Then $\setspec{ t \in \bk }{ \text{$F-t$ is reducible in $A$} } = \{ t_1, \dots, t_{d-1} \}$ and
$$
\text{$F-t_i = G_iH_i$ is the prime factorization of $F-t_i$ in $A$}
$$
for each $i \in \{1, \dots, d-1\}$.
\end{lemma}

\begin{proof}
For each $t \in \bk$,
denote by $n_t$ the number of irreducible components of the closed subset ``$F=t$'' of $\Spec A$.
Then (see, e.g., \cite[1.11]{Dai:GenRatPols}), 
since $\bk$ is algebraically closed and $F$ is a rational polynomial of $A = \kk2$,
\begin{equation}  \label {10292hbfebzvv25674782eruer9} 
\textstyle d-1 = \sum _{t\in \bk} (n_t-1) .
\end{equation}
Let $S = \setspec{ t \in \bk }{ \text{$F-t$ is reducible in $A$} }$ and
note that if $t \in S$ then $n_t>1$
(otherwise $F-t$ would be a power, so $F-s$ would be reducible for all $s \in \bk$,
which would contradict the definition of rational polynomial).
Then formula \eqref{10292hbfebzvv25674782eruer9} implies that
$S = \{ t_1, \dots, t_{d-1} \}$ and
that $n_t=2$ for all $t \in S$.

Fix $i \in \{1, \dots, d-1\}$; let us prove that $G_i,H_i$ are irreducible. Since $n_{t_i}=2$,
there exist irreducible $P_1,P_2 \in A$ such that
$G_i=P_1^{k_1}P_2^{k_2}$ and $H_i=P_1^{l_1}P_2^{l_2}$ with $k_1+k_2\geq 1$ and $l_1+l_2\geq 1$.
Interchanging $P_1,P_2$ if necessary, we may arrange that $k_1 \ge 1$.
Choose $(X,Y) \in \Cgoth(A)$ such that (i) or (ii) holds.
Given $M(X,Y) \in \bk[X,Y]=A$, write $\overline{M}=M(X,0)$ and $\underline{M} = M(0,Y)$.
We have $\overline G_i=\overline P_1^{k_1}\overline P_2^{k_2}$
and similarly for $\underline{G_i}$, $\overline H_i$, $\underline{H_i}$.

Assume that (ii) holds.
Then $\overline G_i\in\bk^*$, $\underline{G_i}\notin\bk$, $\overline H_i\notin\bk$ and $\underline{H_i}\in\bk^*$, so:
$$
\xymatrix @C=3em{
k_1>0
\ar @{=>} [r]^-{\overline G_i\in\bk^*}
& \overline{P_1}\in\bk^* 
\ar @{=>} [r]^-{\overline H_i\notin\bk}
\ar @{=>} [d]_(.45){\overline H_i\notin\bk}
& \overline{P_2}\notin\bk
\ar @{=>} [r]^-{\overline G_i\in\bk^*}
& k_2=0
\\
&  l_2>0  
\ar @{=>} [r]^-{\underline{H_i}\in\bk^*}
& {\underline{P_2}\in\bk^*}
\ar @{=>} [r]^-{\underline{G_i}\notin\bk}
& {\underline{P_1}\notin\bk}
\ar @{=>} [r]^-{\underline{H_i}\in\bk^*}
& l_1=0,
}
$$
so $G_i=P_1^{k_1}$ and $H_i=P_2^{l_2}$.
Since $\underline{G_i}$ and $\overline{H_i}$ are not powers we get $k_1=1=l_2$, so $G_i,H_i$ are irreducible
(in case (ii)).
The argument in case (i) is left to the reader.
\end{proof}

The following examples construct field generators and use \ref{last} to establish  their properties.
The relevance of Example~\ref{26tsgfvbxnbcbhjdhbkizey} is explained in \ref{o832bi87di2837qtduj}.

\begin{example}  \label {26tsgfvbxnbcbhjdhbkizey}
Define $F \in A = \bk[x,y]$ by
\begin{multline*}
F (x,y)= 2+x+27x^9y^3+2y+80x^4y^2+9yx+16yx^2+91x^5y^2+14x^3y \\
+12x^{20}y^7
+ 75x^{17}y^6+196x^{14}y^5+274x^{11}y^4+8x^{28}y^{10}+73x^{25}y^9+296x^{22}y^8 \\
+700x^{19}y^7+1064x^{16}y^6
+ 1078x^{13}y^5+728x^{10}y^4+316x^7y^3+32x^6y^2+8x^{12}y^4 \\
+216x^8y^3+2x^{36}y^{13}+24x^{33}y^{12}
+ 132x^{30}y^{11}+440x^{27}y^{10}+990x^{24}y^9+1584x^{21}y^8 \\
+1848x^{18}y^7+1584x^{15}y^6+990x^{12}y^5+440x^9y^4 + 132x^6y^3+24x^3y^2+2x^4y .
\end{multline*}
The polynomial $F$ is a rectangular polynomial of bidegree $(36,13)$.
We claim:
\begin{equation}  \label {0182x9238729gcbrnxsh}
\textit{$F$ is a very good field generator of $A$ with $\Delta(F,A) = [1,3,4]$.}
\end{equation}

First note that the Newton polygon of $F$ has $3$ faces with face polynomial
$f_1(x,y)=y(x^3y+1)^{12}$, $f_2(x,y)=x^4y(x^8y^3+1)^4$ and $f_3(x,y)=x+x^4y$.
The last one produces a dicritical of degree $1$. We study the two other.

Let $F_h(x,y,z)=(F(x/z,y/z)-t)z^{49}$.

We first consider $F_h(x,1,z)$ and we apply the blowups $x\to xz$ divided by $z^{36}$, $z\to zx^2$ divided by $x^{24}$ and $x \to xz$ divided by $z^{12}$ and then the change 
$x\to x-1$. We get 
$$2x^{12}-x^8z+x^4z^2+(1-t)z^3+\sum a_{\alpha_1,\alpha_2}x^{\alpha_1}z^{\alpha_2}$$
with $\alpha_1+4\alpha_2>12$. Then the face with face polynomial $f_1$ produces a dicritical of degree $3$.

Now consider $F_h(1,y,z)$ and apply $y\to yz^3$ divided by $z^{39}$, $z\to zy$ divided by $y^9$ and $y\to yz^2$ divided by $z^8$. Apply the change $y\to y-1$. We get
$$2y^4-23y^2z+100y^2z^2-194yz^3+(t+140)z^3+\sum b_{\alpha_1,\alpha_2}x^{\alpha_1}z^{\alpha_2}$$
with $\alpha_1+\alpha_2>4$, which proves that we have a dicritical of degree $4$.
So $\Delta(F,A) = [1,3,4]$.
The Newton tree at infinity of this polynomial is shown in Figure~\ref{87234856fyuhduweyr}.

\begin{figure}[ht]
\begin{center}
\includegraphics{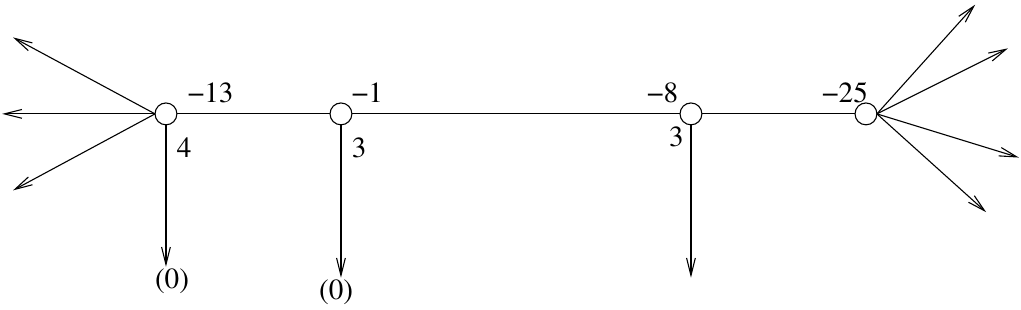}
\caption{}
\label {87234856fyuhduweyr}
\end{center}
\end{figure}  

It is easy to check that $F$ is a rational polynomial (hence a field generator) of $A$.

Two fibers of $F$ are reducible namely, let
$$
h_1(x,y)= x+2+8x^9y^3+8x^3y+12x^6y^2+2x^{12}y^4+2x^4y ,
$$
\begin{multline*}
h_2(x,y)= x^{24}y^9+8x^{21}y^8+28x^{18}y^7+3x^{16}y^6+ 56x^{15}y^6+16x^{13}y^5 \\
+70x^{12}y^5+34x^{10}y^4
+ 56x^9y^4+3x^8y^3+36x^7y^3+28x^6y^3+8x^5y^2 \\
+19x^4y^2+8x^3y^2+6yx^2+4yx+1+y ,
\end{multline*}
$$
h_3(x,y)= x^9y^3+1+x+3x^3y+3x^6y^2 ,
$$ 
\begin{multline*}
h_4(x,y)= 2x^{27}y^{10}+18x^{24}y^9+72x^{21}y^8+6x^{19}y^7+168x^{18}y^7+37x^{16}y^6\\
+ 252x^{15}y^6+95x^{13}y^5+252x^{12}y^5+6x^{11}y^4+130x^{10}y^4+168x^9y^4 + 20x^8y^3 \\
+100x^7y^3+72x^6y^3+23x^5y^2+41x^4y^2+18x^3y^2+2x^3y+9yx^2+7yx+1+2y .
\end{multline*}

We have $F(x,y)=h_1(x,y)h_2(x,y)$ and $F(x,y)-1=h_3(x,y)h_4(x,y)$; then \ref{783e826gehfood0293}(ii) implies
that these are the prime factorizations of $F$ and $F-1$ respectively, and that
$F-t$ is irreducible for all $t \in \bk\setminus\{0,1\}$.
So $\subdeg_\gamma(F) = 12$ (where $\gamma = (x,y)$) and 
$$
\delta(F,A) = 7 > \frac{49}{12} = \frac{\deg_\gamma(F)}{\subdeg_\gamma(F)} .
$$
By \ref{last}, $F$ is a very good field generator of $A$, so \eqref{0182x9238729gcbrnxsh} is proved.
\end{example}

\begin{example}  \label {p92839819ueg87gdywfe6}
Let $A = \bk[v,w] = \kk2$.
Define $P, Q \in A$ by
\begin{multline*}
P=v^4(vw-1)^{14}+w^2(v^3w^2+1)^6-v^{18}w^{14} -70v^{14}w^{11}+374v^{13}w^{10} \\
+(-140v^{11}-1210v^{12})w^9+(588v^{10}+2640v^{11})w^8 \\
+(-140v^8-1464v^9-4092v^{10})w^7 +(430v^7+2394v^8+4620v^9)w^6\\
+(-70v^5-757v^6-2688v^7-3828v^8)w^5 +(2100v^6+137v^4+2310v^7+835v^5)w^4\\
+(-1128v^5-990v^6-590v^4-14v^2-142v^3)w^3 +(260v^3+81v^2+286v^5+399v^4+12v)w^2\\
+(-84v^3-50v^4-24v-65v^2-3)w+3+7v+8v^2+4v^3,
\end{multline*}
\begin{multline*}
Q=v^5(vw-1)^{13}+w(v^3w^2+1)^6-v^{18}w^{13} -65v^{14}w^{10}+320v^{13}w^9 \\
+(-130v^{11}-945v^{12})w^8 +(502v^{10}+1860v^{11})w^7+(-2562v^{10}-130v^8-1134v^9)w^6\\
+(1652v^8+366v^7+2520v^9)w^5+(-578v^6-1770v^8-65v^5-1610v^7)w^4\\
+(552v^5+1050v^6+116v^4+870v^7)w^3+(-318v^4-104v^3-285v^6-13v^2-442v^5)w^2\\+(102v^3+109v^4+46v^2+10v+56v^5)w-1-8v-12v^3-5v^4-14v^2 .
\end{multline*}
Let
$$
F=PQ \in A.
$$
Then $F$ is a rectangular polynomial of degree $63$. 
Its Newton tree at infinity is in Figure 5.
\begin{figure}[ht]
\begin{center}
\includegraphics{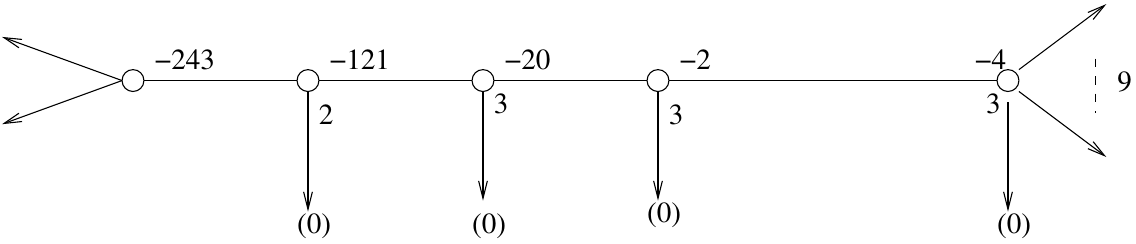}
\caption{}
\label {fg1d2t23e6778fhj12dfsf}
\end{center}
\end{figure}  
It follows that $F$ is a rational polynomial (hence a field generator) of $A$ with $\Delta (F,A)=[9,2]$.
We claim:
\begin{equation} \label {9023ggzxnmxcu9d8ody823}
\textit{$F$ is a very bad field generator of $A$ that is lean in $A$.}
\end{equation}
Indeed, we have $\Galg(F,A)\subseteq \{ (v),(w) \}$ by \ref{09239023r02n02b27c2c8h2}(b);
as $\deg F(v,0)=9$ and $\deg F(0,w)=3$,
\ref{09239023r02n02b27c2c8h2}(a) gives $\Galg(F,A) = \emptyset$.
Then $F$ is a very bad field generator of $A$ by \ref{88x889adb823mdnfv}.

Lemma~\ref{783e826gehfood0293}(i) implies that $F-t$ is irreducible for every $t \in \bk^*$ and that $F=PQ$
is the prime factorization of $F$;
thus $\subdeg_\gamma(F) = \deg_\gamma(Q) = 31$ (where $\gamma=(v,w)$) and 
$$
\delta(F, A ) = 11 > \frac{63}{31} = \frac{\deg_\gamma(F)}{\subdeg_\gamma(F)}.
$$
So \ref{last} implies that $F$ is lean in $A$, i.e., \eqref{9023ggzxnmxcu9d8ody823} is proved.

\end{example}

\begin{example}[Continuation of \ref{2983y872613765xrenucu}]  \label {pf98287217ys82}
We showed in \ref{2983y872613765xrenucu} that the polynomials $F_{CND}$ are bad field generators of $\bk[X,Y]$.
We now prove:
$$
\textit{The polynomials $F_{CND}$ are lean in $\bk[X,Y]$.}
$$

Let $\phi(X)=\prod(X-a_i)^{m_i}$. Define $P(X,Y)=X^{n-1}Y+\tilde{\phi}(X^{n-2}Y).$ We have
$$F_{CND}(X,Y)=\frac{1}{\tilde{\phi}(X^{n-2}Y)}(X+YP^{n-1}(X,Y)(c_{n-1}(\frac{X}{P(X,Y)})^{n-1}+\cdots+c_0))$$
$$F_{CND}(X,Y)=\frac{1}{\tilde{\phi}(X^{n-2}Y)}(X+YP^{n-1}(X,Y)(\phi(\frac{X}{P(X,Y)})-(\frac{X}{P(X,Y)})^n))$$
$$F_{CND}(X,Y)=\frac{Y}{\tilde{\phi}(X^{n-2}Y)P(X,Y)} \prod(X-a_iP(X,Y))^{m_i}-\frac{X}{P(X,Y)}$$
Let $a\in\bk$ be a root of the polynomial $\phi$ (observe that $a \neq 0$).  Then
$$
F_{CND}-a = R_a Q_a
$$
where $R_a = X-aP(X,Y)$ and $Q_a \in \bk[X,Y]$.
Note that $\deg_\gamma(R_a) = n(n-1) < \deg_\gamma(Q_a)$ ($\gamma=(X,Y)$). Also,
\begin{equation*}
\begin{array}{lllll}
F(X,0)-a=X-a &\text{and}  & R_a(X,0)=X-a  & \implies & Q_a(X,0)=1
\\
F(0,Y)-a=c_0Y-a &\text{and} & R_a(0,Y)=-a & \implies & Q_a(0,Y)= (-c_0/a)Y+1 .
\end{array}
\end{equation*}
By \ref{783e826gehfood0293}(ii), we find that $\setspec{ t \in \bk }{ \text{$F_{CND}-t$ is reducible} }$
is the set of roots of $\phi$ and that, for each root $a$, $F_{CND}-a = R_a Q_a$ is
the prime factorization of $F_{CND}-a$.
Consequently, $\subdeg_\gamma(F_{CND}) = n(n-1)$. Then
$$
\delta(F_{CND}, \bk[X,Y]) = 2n-3 \ge \frac{n(n-1)(n-2)+1}{n(n-1)} =
\frac{\deg_\gamma(F_{CND})}{\subdeg_\gamma(F_{CND})} ,
$$
so \ref{last} implies that $F_{CND}$ is lean in $\bk[X,Y]$.
\end{example}

\begin{remark}
Let $A = \kk2$.
Given $F \in A \setminus\bk$, let us use the temporary notation 
$$
m(F,A) = \min( t, d_1, \dots, d_t ) \quad \text{where $\Delta(F,A) = [d_1, \dots, d_t]$}.
$$
Then \ref{2983y872613765xrenucu} and \ref{pf98287217ys82} show that the set
$$
\setspec{ m(F,A) }{ \text{$F$ is a bad field generator of $A$ which is lean in $A$} }
$$
is not bounded.
\end{remark}

\begin{remark}
The polynomial\footnote{Let $g$ be the element of the family $\{F_R\}$ of \ref{872349726edh83uw} obtained
by setting $a_0=a_1=a_2=0$; then $f(x,y)=g(y,x)$.}
$f(x,y)=x((x^2y^3+1)^2+xy^2)^2$ is a bad field generator of degree $21$
in $B=\bk[x,y]$ with $\Delta(f,B) =[2,3]$. It has the same Newton tree as Russell's polynomial.
It is a lean bad field generator such that
$$
\textstyle
\delta(f,B) = 5 < 21 = \frac{ \deg_\gamma(f) }{ \subdeg_\gamma(f) } \quad \text{where $\gamma = (x,y)$}.
$$
Now the same $f$ is a bad field generator of  $B' = k[x,y/x]$ of degree $33$,
with $\Delta(f,B') =[2,3]$, and which is not lean in $B'$.
So there exist bad field generators with two dicriticals and which are not lean,
and there exist lean field generators $f$ of $B$ satisfying
$$
\textstyle
\delta(f,B) < \frac{ \deg_\gamma(f) }{ \subdeg_\gamma(f) } .
$$
One has to note that in those two cases the reducible fiber is not reduced.

\end{remark}

\bigskip
\bigskip

\begin{center}
{\large\bf Part II: Birational endomorphisms of the affine plane}
\end{center}

\bigskip
Throughout Part~II, $\bk$ denotes an algebraically closed field of arbitrary characteristic
and $\aff^2 = \aff^2_\bk$ is the affine plane over $\bk$.

\section{Introduction}
\label {IntroPartII78364174}

By a {\it birational endomorphism of $\aff^2$}, we mean a
birational morphism from $\aff^2$ to $\aff^2$ (cf.\ \ref{xvscgdfhh82927475}).
The set $\Bir( \aff^2 )$ of birational endomorphisms of $\aff^2$
is a monoid under composition of morphisms,\footnote{In fact $\BirA$ is a {\it cancellative\/} monoid,
since it is included in the Cremona group of $\proj^2$.}
and the group of invertible elements
of this monoid is the automorphism group $\Aut(\aff^2)$ of $\aff^2$.
An element $f$ of $\BirA$ is {\it irreducible\/} if it is not invertible and,
for every factorization $f = h \circ g$ of $f$ with $g,h \in \BirA$, one of $g,h$ is invertible.

The birational morphism
$c : \aff^2 \to \aff^2$, $c(x,y) = (x,xy)$, is an example of a non-invertible element of $\BirA$.
One can ask:
\begin{equation} \tag{$*$}
\textit{Is $\AutA \cup \{c\}$ a generating set for the monoid $\BirA$?}
\end{equation}
That question was posed in Abhyankar's seminar at Purdue University in the early 70s,
and was given a negative answer by P.\ Russell, who gave
an example (which appeared later in \cite[4.7]{Dai:bir}) of an irreducible element of $\BirA$
which is not of the form $u \circ c \circ v$ with $u,v \in \AutA$.

The above question $(*)$ and its answer eventually gave rise to further studies of the monoid $\BirA$.
The reader is referred to \cite{CassouDaigBir} for a summary of the state of knowledge and a bibliography 
of this subject.
Our aim, here, is to briefly review some results of \cite{CassouDaigBir} that are directly related to question~$(*)$.
Section~\ref{SEC:Irreducibleelementsandgeneratingsets} gathers some observations that show that if $S$ is
any subset of $\BirA$ such that $\AutA \cup S$ generates $\BirA$, then $S$ has to be large.
Section~\ref{SomepropertiesofAeulinBirA}
reviews what is known about the submonoid $\Aeul$ of $\BirA$ generated by the set $\AutA \cup \{c\}$.

\medskip
Let us give some definitions, notations, and facts.

\medskip
Elements $f,g \in \BirA$ are said to be {\it equivalent\/} ($f \sim g$) if $f = u \circ g \circ v$ for
some $u,v \in \AutA$. This is indeed an equivalence relation on the set $\BirA$, but keep in mind that
the conditions $f \sim f'$ and $g \sim g'$ do not imply that $g \circ f \sim g' \circ f'$.

We write $[f]$ for the equivalence class of an element $f$ of $\BirA$.

Let $f \in \BirA$.
Recall from \ref{xvscgdfhh82927475} that $f$ has finitely many contracting curves and missing curves.
Let $c(f)$ (resp.\ $q(f)$) be the number of contracting (resp.\ missing) curves of $f$.
Clearly, if $f \sim g$ then $c(f)=c(g)$ and $q(f) = q(g)$.
By \cite[4.3(a)]{Dai:bir}, $c(f)=q(f)$ for every $f \in \BirA$.

Consider the map $f \mapsto n(f)$, from $\BirA$ to $\Nat$, defined in \cite[1.2(a)]{Dai:bir}
or \cite[2.3]{CassouDaigBir}.
Then  \cite[2.12]{Dai:bir} shows that
\begin{equation} \label {nm21m1j2uo87f7g4h2aaahklvxtw}
\text{$n : \BirA \to (\Nat,+)$ is a homomorphism of monoids}
\end{equation}
and it is pointed out in \cite[2.6(b)]{CassouDaigBir} that 
\begin{equation} \label {x90zdx7ffyrgytd2}
\AutA = \setspec{ f \in \BirA }{ n(f)=0 } .
\end{equation}
Statements \eqref{nm21m1j2uo87f7g4h2aaahklvxtw} and \eqref{x90zdx7ffyrgytd2} immediately imply that
each non-invertible element of $\BirA$ is a composition of irreducible elements, i.e.,
\begin{equation} \label {12xzx3xc5xv46h5bjyiutuql}
\text{the monoid $\BirA$ has factorizations into irreducibles.}
\end{equation}
Facts (\ref{nm21m1j2uo87f7g4h2aaahklvxtw}--\ref{12xzx3xc5xv46h5bjyiutuql})
were known in the time of \cite{Dai:bir} and \cite{Dai:trees}, but were not stated explicitly.
Essentially nothing is known regarding uniqueness of factorizations.

In view of \eqref{12xzx3xc5xv46h5bjyiutuql}, it is natural to ask whether one can list all irreducible
elements of $\BirA$ up to equivalence.  However, various examples and facts indicate
that $\Bir( \aff^2 )$ contains a great diversity of irreducible elements of arbitrarily high complexity,
and this suggests that the task of finding all of them may be hopeless.
In this regard, let us mention that \cite[4.12]{Dai:bir} implies in particular:
$$
\setspec{ n(f) }{ \text{$f$ is an irreducible element of $\BirA$} } = \Nat\setminus\{0\}.
$$
The results reviewed in Section~\ref{SEC:Irreducibleelementsandgeneratingsets} strengthen the
impression that there are too many irreducible elements to describe them.
It therefore makes sense to turn our attention, as we do in Section~\ref{SomepropertiesofAeulinBirA},
to other types of questions regarding the structure of the monoid $\BirA$.

\section{Irreducible elements and generating sets}
\label {SEC:Irreducibleelementsandgeneratingsets}

The results of this section show that question $(*)$, stated in the introduction to Part~II,
has a ``very negative'' answer.
The first three results are 4.1--4.3 of \cite{CassouDaigBir}.
The proof of the first one is based on Example 4.13 of \cite{Dai:bir},
which constructs a family of irreducible elements of $\BirA$; the proof shows that that 
family already contains $|\bk|$ nonequivalent irreducible elements of $\BirA$.
The vertical bars denote cardinality (recall that $\bk$ is algebraically closed, so $|\bk|$
is an infinite cardinal).

\begin{lemma}[{\cite[4.1]{CassouDaigBir}}] \label {.2938f9283w9fq0349hfio}
$ \big| \setspec{ [f] }{ \text{$f$ is an irreducible element of $\BirA$} } \big| = |\bk|$. 
\end{lemma}

In the next result, (i) $\Rightarrow$ (ii) is clear and the converse easily follows from
the fact that the monoid $\BirA$ has factorizations into irreducibles.

\begin{lemma}[{\cite[4.2]{CassouDaigBir}}] \label {92389h32fhfqfqasde}
For any subset $S$ of $\BirA$, the following are equivalent:
\begin{enumerate}

\item[(i)] $\AutA \cup S$ is a generating set for the monoid $\BirA$;
\item[(ii)] for each irreducible $f \in \BirA$, $[f] \cap S \neq \emptyset$.
\end{enumerate}
\end{lemma}

The next fact is an immediate consequence of \ref{.2938f9283w9fq0349hfio} and \ref{92389h32fhfqfqasde}:

\begin{corollary}[{\cite[4.3]{CassouDaigBir}}] \label {2938923d89h209jdhvas.}
Let $S$ be a subset of $\BirA$ such that $\AutA \cup S$ is a generating set for the monoid $\BirA$.
Then $|S| = |\bk|$.
\end{corollary}

A question posed by Patrick Popescu-Pampu asks whether one can find a set $S \subset \BirA$ such that
$\AutA \cup S$ generates $\BirA$ and the elements of $S$ have bounded degree.
Result \ref{89j32467ewqfds}, below, gives a negative answer to 
that question.  First, we define what we mean by the degree of an element of $\BirA$.
See the general introduction for the notions of coordinate system of $\kk2$ or $\aff^2$,
and for the degree(s) of an element of $\kk2$.
Let $\Cgoth$ denote the set of coordinate systems of $\aff^2$.

\begin{definition}
Let $f \in \BirA$.
\begin{enumerata}

\item Let $\gamma=(X,Y) \in \Cgoth$. 
Using $\gamma$ to identify $\aff^2$ with $\bk^2$, we can consider that
$f : \aff^2 \to \aff^2$ is given by $f(x,y) = (u(x,y), v(x,y))$ for some unique polynomials $u,v \in \bk[X,Y]$.
We define $\deg_\gamma f = \max( \deg_\gamma u,  \deg_\gamma v )$. 

\item Define $\DEG f = \min\setspec{ \deg_\gamma f }{ \gamma \in \Cgoth }$ and
$$
\UDEG f = \min\setspec{ \DEG g }{ g \in [f] }
=  \min\setspec{ \deg_\gamma g }{ g \in [f] \text{ and } \gamma \in \Cgoth }.
$$

\end{enumerata}
\end{definition}

Note that $1 \le \UDEG f \le \DEG f \le \deg_\gamma f \in \Nat$, for every $f \in \BirA$ and $\gamma \in \Cgoth$.

\begin{corollary} \label {89j32467ewqfds}
Let $S$ be a subset of $\BirA$ such that $\AutA \cup S$ is a generating set for the monoid $\BirA$.
Then $\setspec{ \UDEG f }{ f \in S }$ is not bounded.
\end{corollary}

\begin{smallremark}
This is a slight improvement of Corollary~4.5 of \cite{CassouDaigBir},
which states that $\setspec{ \DEG f }{ f \in S }$ is not bounded.
The following proof is a small modification of that of \cite[4.5]{CassouDaigBir},
and inequality \eqref{78df2hwe65ghds} slightly improves  \cite[4.4]{CassouDaigBir}.
\end{smallremark}

\begin{proof}[Proof of \ref{89j32467ewqfds}]
Remark 4.4 of \cite{CassouDaigBir} shows that $\DEG g \ge (c(g)+2)/2$ for every $g \in \BirA$,
where $c(g)$ is the number of contracting curves of $g$.
Given $f \in \BirA$, we may choose $g \in [f]$ such that $\UDEG f = \DEG g$; since 
$c(g)=c(f)$, it follows that  $\UDEG f = \DEG g \ge (c(g)+2)/2 =  (c(f)+2)/2$. We showed:
\begin{equation} \label {78df2hwe65ghds}
\UDEG f \ge \frac{c(f)+2}2 \quad \text{for every $f \in \BirA$}.
\end{equation}

Now let $S$ be as in the statement, and let $n \in \Nat$.
By \cite[4.13]{Dai:bir}, there exists an irreducible element $g \in \BirA$ satisfying $c(g) \ge 2n$.
By \ref{92389h32fhfqfqasde}, there exists $f \in S$ satisfying $f \sim g$; then $c(f) = c(g) \ge 2n$,
so $\UDEG f > n$ by \eqref{78df2hwe65ghds}.
\end{proof}

\section{Some properties of $\Aeul$ in $\BirA$}
\label {SomepropertiesofAeulinBirA}

Let $\Cgoth$ denote the set of coordinate systems of $\aff^2$.

Let $\gamma \in \Cgoth$, 
use $\gamma$ to identify $\aff^2$
with $\bk^2$, and define an element $c_\gamma \in \BirA$ by $c_\gamma(x,y) = (x,xy)$.
Then the equivalence class $[c_\gamma]$ is independent of the choice of $\gamma$;
the elements of $[c_\gamma]$ are called {\it simple affine contractions} (SAC).

By \eqref{nm21m1j2uo87f7g4h2aaahklvxtw} and \eqref{x90zdx7ffyrgytd2},
all elements of $\setspec{ f \in \BirA }{ n(f)=1 }$ are irreducible.
Now \cite[4.10]{Dai:bir} implies:
\begin{equation} \label {9823gs6543hgr26tr4653r254}
\text{$\setspec{ f \in \BirA }{ n(f)=1 }$ is the set of simple affine contractions.}
\end{equation}
So SACs are irreducible and, in fact,
SACs are the simplest irreducible elements and the simplest non-invertible elements of $\BirA$.

Let $\Aeul$ be the submonoid of $\BirA$ generated by automorphisms and simple affine contractions.
Equivalently, 
given any $\gamma \in \Cgoth$, we may describe $\Aeul$ by:
$$
\text{$\Aeul =$ the submonoid of $\BirA$ generated by $\AutA \cup \{c_\gamma\}$.}
$$
One has $\Aeul \neq \BirA$ by \ref{2938923d89h209jdhvas.}, or because question~$(*)$
(in the introduction of Part~II) has a negative answer.

\begin{definition}
Let $\Meul_0$ be a submonoid of a monoid $\Meul$.
We say that $\Meul_0$ is {\it factorially closed\/} in $\Meul$ if the conditions $x,y \in \Meul$
and $xy \in \Meul_0$ imply that $x,y \in \Meul_0$.
\end{definition}

For instance, note the following trivial fact:

\begin{lemma} \label {89dnyt54bjhzshi374}
Let $\Meul$ be a monoid and $\delta : \Meul \to (\Nat,+)$ a homomorphism of monoids.
Then $\setspec{ x \in \Meul }{ \delta(x) = 0 }$ is factorially closed in $\Meul$.
\end{lemma}

By \eqref{nm21m1j2uo87f7g4h2aaahklvxtw}, \eqref{x90zdx7ffyrgytd2} and \ref{89dnyt54bjhzshi374},
we see that $\AutA$ is factorially closed in $\BirA$.
However, the following question turns out to be much more difficult:
\begin{equation} \tag{$**$}
\textit{Is $\Aeul$ factorially closed in $\BirA$?}
\end{equation}
There are several reasons why $(**)$ is a natural question, for instance its relation with
the open question of uniqueness of factorizations in $\BirA$:

\begin{lemma}
Suppose that $\BirA$ has the following property:
\begin{equation} \label {u65435467647yrhefdj}
\begin{minipage}[t]{.9\textwidth}
For any irreducible elements $f_1, \dots, f_r, g_1, \dots, g_s$ of $\BirA$ satisfying\\
$f_1 \circ \cdots \circ f_r = g_1 \circ \cdots \circ g_s$, we have $r=s$.
\end{minipage}
\end{equation}
Then $\Aeul$ is factorially closed in $\BirA$.
\end{lemma}

\begin{proof}
Suppose that \eqref{u65435467647yrhefdj} is true.
Then we may define a homomorphism of monoids, $\ell : \BirA \to (\Nat,+)$,
by stipulating that if $f_1, \dots, f_r$ are irreducible elements of $\BirA$ then
$\ell(f_1 \circ \cdots \circ f_r) = r$, and $\ell(f)=0$ for all $f \in \AutA$.
Define $\delta(f) = n(f)-\ell(f)$ for $f \in \BirA$;
then \eqref{x90zdx7ffyrgytd2} implies that 
$\delta(f) = n(f)-1 \ge 0$ for each irreducible $f \in \BirA$,
and that $\delta(f) = 0$ for each $f \in \AutA$; thus
$$
\text{$\delta : \BirA \to (\Nat,+)$ is a homomorphism of monoids.}
$$
We have $\setspec{ f \in \BirA }{ n(f) \le 1 } \subseteq \Aeul \subseteq \setspec{ f \in \BirA }{ \delta(f) = 0 }$
by \eqref{x90zdx7ffyrgytd2} and \eqref{9823gs6543hgr26tr4653r254}.
If $f \in \BirA$ satisfies $\delta(f) = 0$ then either $f \in \AutA$ or $f = f_1 \circ \cdots \circ f_r$ for some
irreducible elements $f_1, \dots, f_r$ of $\BirA$; in the first case it is clear that $f \in \Aeul$,
and in the second case we have (for each $i$) $0 = \delta(f_i) = n(f_i)-1$, so $f_i \in \Aeul$;
so $f\in \Aeul$ in both cases, showing that  $\Aeul = \setspec{ f \in \BirA }{ \delta(f) = 0 }$.
The desired conclusion follows from~\ref{89dnyt54bjhzshi374}.
\end{proof}

Because we don't know whether $\BirA$ has property \eqref{u65435467647yrhefdj},
it is interesting to see that $\Aeul$ is indeed factorially closed in $\BirA$:

\begin{theorem}[{\cite[4.8]{CassouDaigBir}}]  \label {9kjxt12544d2jhdbhdfa384732}
If $f,g \in \BirA$ satisfy $g\circ f \in \Aeul$, then $f,g \in \Aeul$.
\end{theorem}

\medskip
We want to mention another result of  \cite{CassouDaigBir} related to $\Aeul$.
It is customary to define families of elements of $\BirA$ by requiring that their missing curves
(or sometimes their contracting curves) satisfy some condition or other.
For instance, the introduction of Section~3 of \cite{CassouDaigBir}
defines three subsets 
$S_{\text{\rm w}} \supset S_{\text{\rm a}} \supset S_{\text{\rm aa}}$ of $\BirA$ by that method.
Let us consider in particular $\Sw$, which is defined to be 
the set of $f \in \BirA$ satisfying:
\begin{quote}
there exists a coordinate system of $\aff^2$ with respect to which all missing curves of $f$ have degree $1$.
\end{quote}

Note that $\AutA \cup \setspec{ c_\gamma }{ \gamma \in \Cgoth } \subseteq \Sw$.

\begin{example}
Choose a coordinate system of $\aff^2$ and use it to identify $\aff^2$ with $\bk^2$.
Define $c,\theta,f : \aff^2 \to \aff^2$ by $c(x,y) = (x,xy)$, $\theta(x,y) = (x+y^2-1,y)$ and $f = c \circ \theta \circ c$.
As $c \in \setspec{ c_\gamma }{ \gamma \in \Cgoth }$ and $\theta \in \AutA$, we have
$c,\theta \in \Aeul$ and hence $f \in \Aeul$.
The singular curve $y^2= x^2+x^3$ is a missing curve of $f$, so $f \notin \Sw$ and hence $\Aeul \nsubseteq \Sw$.
As $c,\theta \in \Sw$,
this also shows that $\Sw$ is not closed under composition of morphisms.
\end{example}

\begin{smallremark}
Russell constructed an example 
(which appeared later in \cite[4.7]{Dai:bir}) of an
element $f$ of $\Bir( \aff^2 )$ with three missing curves $C_1, C_2, C_3$,
where (with respect to a suitable coordinate system of $\aff^2$) $C_1$ and $C_2$ 
are the lines $x+y=0$ and $x-y=0$, and $C_3$ is the parabola $y=x^2$.
Note that for each $i \in \{1,2,3\}$, there exists a coordinate system of $\aff^2$ with respect
to which $C_i$ has degree $1$. However, since $C_1 \cap C_3$ consists of two points,
no coordinate system $\gamma$ of $\aff^2$ has the property that
$\deg_\gamma C_1 = \deg_\gamma C_2 = \deg_\gamma C_3 = 1$. So $f \notin \Sw$.
This shows that, in order for $f$ to belong to $\Sw$, it is not enough that each individual missing curve
be isomorphic to a line; the correct condition is that the missing curves be ``simultaneously rectifiable''.
\end{smallremark}

As a consequence of Theorem 3.15 of \cite{CassouDaigBir},  one has:

\begin{corollary}
The set $\Sw$ is included in $\Aeul$.
\end{corollary}

Note that Theorem 3.15 of \cite{CassouDaigBir} gives a complete description of the three subsets
$S_{\text{\rm w}} \supset S_{\text{\rm a}} \supset S_{\text{\rm aa}}$ of $\Aeul$, and that it does so by describing
which compositions of automorphisms and simple affine contractions give elements of each of these sets.

\end{document}